\documentclass[12pt]{amsart}
\usepackage{amsmath,amsthm,amsfonts,amssymb,latexsym,enumerate,xcolor} 
\usepackage[pagebackref]{hyperref} 
\usepackage{float}
\usepackage{graphicx}

\newcommand{\Q}{\mathbb{Q}}
\newcommand{\Z}{\mathbb{Z}}

\newcommand{\Syl}{\operatorname{Syl}}

\newcommand{\Irr}{\operatorname{Irr}}

\newcommand{\Sym}{\operatorname{Sym}}

\newcommand{\Id}{{{\operatorname{Id}}}}

\newcommand{\Stab}{\operatorname{Stab}}
\newcommand{\I}{\operatorname{I}}

\newcommand{\Sub}{\operatorname{\mathbf{Sub}}}
\newcommand{\fix}{\operatorname{fix}}
\newcommand{\Sup}{\operatorname{Sup}}

\newcommand{\Lin}{\operatorname{Lin}}

\newtheorem{thm}{Theorem}[section]
\newtheorem{lem}[thm]{Lemma}

\newtheorem{pro}[thm]{Proposition}
\newtheorem{cor}[thm]{Corollary}

\newtheorem*{conA}{Conjecture A}
\newtheorem*{conB}{Conjecture B}

\theoremstyle{definition}

\numberwithin{equation}{section}

\newtheorem{thml}{Theorem}

\begin{document}

\title[The Picky and Subnormalizer Conjectures for symmetric groups]{The Picky and Subnormalizer Conjectures for symmetric groups}


\author{Juan Mart\'inez Madrid}
\address{Departament de Matem\`atiques, Universitat de Val\`encia, 46100
  Burjassot, Val\`encia, Spain}
\email{Juan.Martinez-Madrid@uv.es}

\thanks{The research of the  author was supported by Ministerio de
Ciencia e Innovaci\'on (Grant PID2022-137612NB-I00 funded by
MCIN/AEI/10.13039/501100011033 and ``ERDF A way of making Europe"),
and by Generalitat Valenciana  CIACIF/2021/228 and CIBEFP/2023/125.}



\date{\today}

\begin{abstract}
A new type of conjectures on characters of finite groups, related to the McKay conjecture, have recently been proposed. In this paper, we study these conjectures for symmetric groups.
\end{abstract}

\maketitle



\section{Introduction}\label{Section1}

All groups considered in this work are finite. Let $p$ be a prime, let $G$ be a group, and let $x$ be a $p$-element of $G$. Following \cite{MRPicky}, we say that $x$ is a \textbf{picky} element of $G$ if it lies in a unique Sylow $p$-subgroup of $G$. Picky elements were extensively studied in \cite{MMM}, where they were shown to possess desirable properties from the point of view of character theory (see Corollary 2.4 of \cite{MMM}).

This motivated the statement new local-global conjectures involving picky elements and character zeros. Given $x\in G$, we define
\[
\Irr^x(G) = \{ \chi \in \Irr(G) \mid \chi(x) \ne 0 \}.
\]
The following conjecture was formulated in \cite{MRPicky}.

\begin{conA}[Picky Conjecture]
Let $p$ be a prime, let $G$ be a group, and let $P \in \Syl_p(G)$. If $x \in P$ is picky, then there exists a bijection
\[
\Gamma: \Irr^x(G) \rightarrow \Irr^x(\mathbf{N}_G(P))
\]
satisfying the following properties:
\begin{itemize}
    \item[(I)] $\Gamma(\chi)(1)_p = \chi(1)_p$ for every $\chi \in \Irr^x(G)$.
    \item[(II)] $\Q(\Gamma(\chi)(x)) = \Q(\chi(x))$ for every $\chi \in \Irr^x(G)$.
\end{itemize}
\end{conA}

Let us also recall the celebrated McKay Conjecture, which asserts that
\[
|\Irr_{p'}(G)| = |\Irr_{p'}(\mathbf{N}_G(P))|,
\]
where $\Irr_{p'}(G) = \{ \chi \in \Irr(G) \mid \gcd(p, \chi(1)) = 1 \}$. This conjecture was recently proved by Cabanes and Späth \cite{CS}. It follows from Corollary 4.20 of \cite{N} that if $\chi \in \Irr_{p'}(G)$ and $x$ is a $p$-element of $G$, then $\chi(x)\ne 0$. Thus, $\Irr_{p'}(G) \subseteq \Irr^x(G)$ for any $p$-element $x$. Consequently, if $x$ is a picky $p$-element and the Picky Conjecture holds for $(G,x)$, then the McKay Conjecture follows for $G$.

Substantial progress has been made on the Picky Conjecture. For instance, Moretó, Navarro, and Rizo \cite{MNR} have proved a strong version of the conjecture for $p$-solvable groups when $p$ is odd. Malle and Schaeffer Fry \cite{Malle,MalleFry} have verified the conjecture for several families of groups of Lie type.

Our first main result establishes a strong version of the Picky Conjecture for symmetric groups. We introduce the following notation. Given a prime $p$ and $P \in \Syl_p(G)$, define
\[
\mathcal{P} = \{ x \in P \mid x \text{ is picky in } G \},
\]
and let
\[
\Irr^{\mathcal{P}}(G) = \{ \chi \in \Irr(G) \mid \chi(x) \ne 0 \text{ for some } x \in \mathcal{P} \}.
\]

\begin{thml}\label{thmA}
Let $p$ be a prime, let $n \geq 1$ be an integer, and let $P \in \Syl_p(\mathsf{S}_n)$. Then there exists a bijection
\[
\Gamma: \Irr^{\mathcal{P}}(\mathsf{S}_n) \rightarrow \Irr^{\mathcal{P}}(\mathbf{N}_{\mathsf{S}_n}(P))
\]
satisfying:
\begin{itemize}
    \item[(I)] $\Gamma(\chi)(1)_p = \chi(1)_p$ for every $\chi \in \Irr^{\mathcal{P}}(\mathsf{S}_n)$.
    \item[(II)] $\Gamma(\chi)(x) = \pm \chi(x)$ for every $\chi \in \Irr^{\mathcal{P}}(\mathsf{S}_n)$ and $x \in \mathcal{P}$.
    \item[(III)] $\Gamma(\Irr^x(\mathsf{S}_n)) = \Irr^x(\mathbf{N}_{\mathsf{S}_n}(P))$ for every $x \in \mathcal{P}$.
\end{itemize}
\end{thml}

Note that $\mathcal{P} \ne \varnothing$ by Theorem 4.2 of \cite{MMM}. We emphasize that condition (II) does not hold in general for arbitrary groups. Following the proofs of Lemma \ref{TypeI} and Theorems \ref{LastCase} and \ref{OddCase}, one can explicitly compute the values $\chi(x)$ up to sign for any picky element $x \in \mathsf{S}_n$ and any $\chi \in \Irr^x(\mathsf{S}_n)$ (or $\chi \in \Irr^x(\mathbf{N}_{\mathsf{S}_n}(P))$).

There exists a more general version of the Picky Conjecture for arbitrary $p$-elements. To formulate it, we introduce a further concept.

Given $x \in G$, define the {\bf subnormalizer} of $x$ in $G$ by
\[
\Sub_G(x) = \langle g \in G \mid \langle x \rangle \trianglelefteq \trianglelefteq \langle x, g \rangle \rangle.
\]
As shown in \cite{MRPicky}, using results from \cite{Casolo1}, a $p$-element $x \in P$ is picky if and only if $\Sub_G(x) = \mathbf{N}_G(P)$. This leads to the following conjecture.

\begin{conB}[Subnormalizer Conjecture]
Let $p$ be a prime, let $G$ be a group, and let $x \in G$ be a $p$-element. Then there exists a bijection
\[
\Gamma: \Irr^x(G) \rightarrow \Irr^x(\Sub_G(x))
\]
satisfying:
\begin{itemize}
    \item[(I)] $\Gamma(\chi)(1)_p = \chi(1)_p$ for every $\chi \in \Irr^x(G)$.
    \item[(II)] $\Q(\Gamma(\chi)(x)) = \Q(\chi(x))$ for every $\chi \in \Irr^x(G)$.
\end{itemize}
\end{conB}

Progress on this stronger conjecture has been limited, in part due to the difficulty of describing subnormalizers explicitly. In particular, the conjecture remains open for $p$-solvable groups, although some results have been obtained for groups of Lie type by Malle \cite{Malle,Malle2}.

Our second main result proves a strong form of the Subnormalizer Conjecture for symmetric groups when $p = 2$.

\begin{thml}\label{thmB}
Let $n \geq 1$ and let $x \in \mathsf{S}_n$ be a $2$-element. Then there exists a bijection
\[
\Gamma: \Irr^x(\mathsf{S}_n) \rightarrow \Irr^x(\Sub_{\mathsf{S}_n}(x))
\]
satisfying:
\begin{itemize}
    \item[(I)] $\Gamma(\chi)(1)_2 = \chi(1)_2$ for every $\chi \in \Irr^x(\mathsf{S}_n)$.
    \item[(II)] $\Gamma(\chi)(x) = \pm \chi(x)$ for every $\chi \in \Irr^x(\mathsf{S}_n)$.
\end{itemize}
\end{thml}

An important step in the proof of Theorem B is to determine the structure of subnormalizers of $2$-elements of the symmetric groups (see Theorems \ref{Subnormalizers} and \ref{Reduction2}).

The case of odd primes in symmetric groups remains open and is the subject of ongoing work. It will require new techniques beyond those used in this paper.

Stronger versions of the Picky and Subnormalizer Conjectures, reflecting the structure of Theorems \ref{thmA} and \ref{thmB}, will be presented in \cite{MRPicky}.

The rest of the paper is structured as follows. In Section \ref{Preliminaries} we recall the necessary background and establish notation. Section \ref{SubnormalPart} is devoted to the proof of Theorem \ref{thmB}, while Section \ref{PickyPart} contains the proof of Theorem \ref{thmA}.

Our approach makes use of several techniques, including the $p$-core tower. These were the key tool in  MacDonald's \cite{MP} proof of the McKay Conjecture for symmetric groups.


\section{Preliminaries}\label{Preliminaries}

\subsection{Character theory of symmetric groups}

The general definitions and results about the character theory in $\mathsf{S}_n$ have been taken from \cite{JK}. Let $n\geq 1$ be an integer. A partition of $n$ is a sequence $\lambda=(\lambda_1,\ldots,\lambda_k)$ of positive integers such that $\lambda_{i}>\lambda_{i+1}$ and $\sum_{i=1}^k\lambda_i=n$. Given a partition $\lambda$, we can associate a Young diagram to $\lambda$ by using $\lambda_1$ boxes in the first row, $\lambda_2$ boxes in the second row and so on. 

Since two elements of $\mathsf{S}_n$ are conjugate if and only if they have the same cycle type, we deduce that there exists a $1$ to $1$ map from the set of partitions of $n$ to the set of conjugacy classes of $\mathsf{S}_n$ and also to the set of irreducible characters of $\mathsf{S}_n$. Moreover, given a partition $\lambda$ of $n$, we have that there exists a ``canonical'' irreducible character $\chi^{\lambda}\in \Irr(\mathsf{S}_n)$. In addition, $\chi^{\lambda}(1)$ can be calculated from  the Young diagram of $\lambda$. The hook number of a box $b$ of the Young diagram of $\lambda$ is  the sum of the number of boxes below the box $b$, plus the number of boxes to the right of $b$, plus $1$ (for the box $b$ itself). 

The hook length formula, proved by Frame, Robinson and Thrall \cite{FRT}, asserts that 
$$\chi^{\lambda}(1)=\frac{n!}{\prod_{i,j}H^{\lambda}(i,j)},$$
where $H^{\lambda}(i,j)$ denotes the hook number of the entry $(i,j)$ of the Young diagram of $\lambda$.  We will write $H(\lambda)$ to denote the denominator in the above quotient.

For example,  for $\lambda=(4,3,1)$, we have that 
$$\chi^{\lambda}(1)=\frac{8!}{6\cdot 4\cdot 3\cdot 1\cdot4\cdot 2 \cdot 1\cdot 1}=70.$$

The hook length formula provides an algorithm for computing $\chi^{\lambda}(1)$. For a general $x \in \mathsf{S}_n$ it is possible to calculate $\chi^{\lambda}(x)$ recursively by applying the Murnaghan--Nakayama rule (see 2.4.7 of \cite{JK}). Let $\lambda$ be a partition of an integer $n$ and let $q>1$ be an integer. A $q$-hook of $\lambda$ is a hook $(i,j)$ in the Young diagram of $\lambda$ with $H^{\lambda}(i,j)=q$. Let $\tau$ be the $q$-hook corresponding to the boxes below $(i,j)$, the boxes to the right  of $(i,j)$ and the box $(i,j)$ itself. We write $\lambda\setminus \tau$ to denote the partition (of $n-q$) obtained by removing $\tau$ form the Young diagram of $\lambda$. We recall that the height, $h(\tau)$ is defined as $1$ less than the  number of rows of $\tau$.

\begin{thm}[Murnaghan--Nakayama rule]
Let $0<q\leq n$ be integers. Suppose that $\rho\pi \in \mathsf{S}_n$, where $\rho$ is an $q$-cycle and $\pi$ is a permutation of the remaining $n-q$ numbers. If $\lambda$ is a partition of $n$, then 
$$\chi^{\lambda}(\rho\pi )=\sum_{\tau}(-1)^{h(\tau)}\chi^{\lambda\setminus \tau}(\pi),$$
where the sum runs over all  $q$-hooks $\tau$ in $\lambda$.
\end{thm}

We remark that if $\lambda$ does not possess any $q$-hook, then $\chi^{\lambda}(\rho\pi )=0$. Now, we introduce the $p$-quotient and $p$-core of a partition. We will follow the construction provided by \cite{MB}  (see pages 12 and 13 of that book).

 If we remove the $q$-hooks successively (no matter the order for removing) till we arrive at a partition without $q$-hooks, then we obtain the $q$-core of $\lambda$, which we denote by $C_{q}(\lambda)$. It is possible to prove that the $q$-core of $\lambda$ is uniquely determined by $\lambda$ and $q$. We will say that $\lambda$ is a $q$-core if $\lambda=C_q(\lambda)$.

Let $n$ be an  integer and let $m\geq n$ be a fixed integer. Let $\lambda=(\lambda_1,\ldots,\lambda_m)$ with $\lambda_1\geq \lambda_2\geq \ldots\geq \lambda_m\geq 0$ be a partition of $n$ (we add 0's among the $\lambda_i$'s to ensure that  the length is $m$). We define $\delta_m:=(m-1,m-2,\ldots,1,0)$ and $\varepsilon=\lambda+\delta_m$. We notice that $\varepsilon_{i+1}>\varepsilon_{i}$ for every $i$. Let $q>1$ be an integer. For a fixed $r\in \{0,\ldots, q-1\}$, let us assume that $\varepsilon_{i_1}>\ldots>\varepsilon_{i_{m_r}}$ are all the $\varepsilon_i$ such that $\varepsilon_i\equiv r \pmod q$ ($m_r$ is $0$ when there are none of them). We write these numbers as $q\cdot\varepsilon_{k}^{r}+r$ for $\varepsilon_{1}^{r}>\varepsilon_{2}^{r}>\ldots > \varepsilon_{m_r}^{r}$. We define $\lambda_{k}^{r}:=\varepsilon_{k}^{r}-m_r+k$ and $\lambda^r=(\lambda_1^r,\ldots, \lambda_{m_r}^r)$. The collection $(\lambda^0,\lambda^1,\ldots ,\lambda^{q-1})$ is the $q$-quotient of $\lambda$.

It is also possible to calculate the $q$-core of $\lambda$ by using these numbers. Let us consider the $m$ numbers $q\cdot s+r$ for $0\leq s\leq m_{r}-1$ and $0\leq r \leq q-1$. Let us arrange these numbers in descending order $\gamma_1>\gamma_2>\ldots > \gamma_m$. Then we define $\tilde{\lambda}_i=\gamma_i-m+i$ for $1 \leq i \leq m$. It is possible to see that $C_q(\lambda)=(\tilde{\lambda}_1,\ldots, \tilde{\lambda}_m)$.

We notice that $\lambda$ is uniquely determined by its $q$-core and its $q$-quotient. In addition,
$$|\lambda|=|C_q(\lambda)|+q(\sum_{r=0}^{q-1}|\lambda^{r}|).$$

Let $\lambda$ be a partition and let $p$ be a prime.  We define $\lambda_{00}:=C_p(\lambda)$. Given $r\in \{0,\ldots,p-1\}$, the partition $\lambda^r$ can be uniquely determined by its $p$-core $\lambda_{1 r}$ and its $p$-quotient $(\lambda^{r,0}, \ldots , \lambda^{r,p-1})$. We can iterate this process to obtain a set of $p$-cores $T_C(\lambda):=\{\lambda_{ij}\mid0\leq i, 0\leq j \leq p^i-1\}$, which will be called the $p$-core tower of $\lambda$. Assume that $n=\sum_{k= 0}^fa_kp^k$ is the $p$-adic expansion of $n$.  It is possible to see that $\lambda_{ij}=(0)$ for any $i >f$. Since any partition is completely determined by its $p$-core and its $p$-quotient, we have that any partition is uniquely determined by its $p$-core tower.  It follows that $T_C$ defines a bijection from the set of partitions   to the set of $p$-core towers.  Moreover, removing a $p^f$-hook  from the Young diagram of $\lambda$ is equivalent to removing one box in the Young diagram of $\lambda_{fj}$ for some $0\leq j\leq p^f-1$.

Let $p$ be a prime and let $m$ be an integer. We write $\nu_p(m)=a$ if $p^a$ divides $m$ but $p^{a+1}$ does not  divide $m$. For each $k$, we define $b_k= \sum_{j=0}^{p^k-1}|\lambda_{k,j}|$. It is possible to prove that $n = \sum_{k \geq 0}b_k p^k$. Then, following the discussion in \cite{MP}, we have 
$$\nu_p(n!)=\frac{1}{p-1}(n-\sum_{k\geq 0}a_k)$$
and 
$$\nu_p(H(\lambda))=\frac{1}{p-1}(n-\sum_{k\geq 0}b_k).$$
Thus, we have 
\begin{equation}
\label{FormulapPart}
\nu_p(\chi^{\lambda}(1))=\nu_p(n!)-\nu_p(H(\lambda))=\frac{1}{p-1}(\sum_{k\geq 0}b_k-\sum_{k\geq 0}a_k).
\end{equation}

It follows that $\chi^{\lambda}$ is a $p'$-degree character if and only if $\sum_{k\geq 0}b_k=\sum_{k\geq 0}a_k$.  In fact, it was proved in \cite{MP} that this happens if and only if $a_k=b_k$ for any $k \geq 1$.  From this discussion, we can deduce the following results.

\begin{thm}\label{2PrimeDegree}
Let $n$ be an integer and let $n=\sum_{i=1}^{f}2^{n_i}$ with $0\leq n_1<n_2\ldots <n_f$ be the $2$-adic expansion of $n$.  Then $|\Irr_{2'}(\mathsf{S}_n)|= 2^{\sum_{i=1}^fn_i}$.
\begin{proof}
This is Corollary 1.3 of \cite{MP}.
\end{proof}
\end{thm}

\begin{pro}
Let $p$ be a prime and let $n$ be an integer.  Assume that $n=\sum_{i=1}^{f}a_ip^{k_i}$ with $0\leq k_1<k_2\ldots <k_f$ and $1\leq a_i\leq p-1$ for all $i$. Let $\lambda$ be a partition of $n$. Then $\chi^{\lambda}\in \Irr_{p'}(\mathsf{S}_n)$ if and only if $|C_{p}(\lambda)|=n-a_fp^{k_f}$ and $\chi^{C_{p}(\lambda)}\in \Irr_{p'}(\mathsf{S}_{n-a_fp^{k_f}})$. 
\begin{proof}
This is Lemma 3.2 of \cite{GLLV}.
\end{proof}
\end{pro}

From this result, we deduce the following corollary.

\begin{cor}\label{hooks}
Let $k \geq 1$, let $n=p^k$ and let $\lambda$ be a partition of $p^k$. Then $\chi^{\lambda}\in \Irr_{p'}(\mathsf{S}_{p^k})$ if and only if $\lambda$ is a $p^k$-hook. In particular, $|\Irr_{p'}(\mathsf{S}_{p^k})|=p^k$.
\end{cor}

Given $k \geq 1$ and $0 \leq a \leq 2^k-1$, we define $\tau(a,k)$ as the following $2^k$-hook $$\tau(a,k):=(2^k-a,1^a).$$  By Corollary \ref{hooks}, we have that $\Irr_{2'}(\mathsf{S}_{2^k})=\{\chi^{\tau(a,k)}\mid 0 \leq a \leq 2^k-1\}$.

Given $k \geq 1$ and $0\leq a<b\leq 2^{k-1}-1$, we define $\gamma(a,b,k)$ as the partition of $2^k$ given by 
$$\gamma(a,b,k):=(2^{k-1}-a,2^{k-1}-b+1,2^a,1^{b-a-1}).$$ This partition has $\tau(a+1,k-1)$ and $\tau(b-1,k-1)$ as $2^{k-1}$-hooks. Moreover, $\gamma(a,b,k)\setminus \tau(a+1,k-1)=\tau(b,k-1)$ and $\gamma(a,b,k)\setminus \tau(b-1,k-1)=\tau(a,k-1)$. Moreover, the following result shows that these partitions are the unique partitions of $2^k$ with $2$ different $2^{k-1}$-hooks.

\begin{pro}\label{Doublehook}
Let $k \geq 1$ and let $n=2^k$. Let $\lambda$ be a partition of $2^k$. Then $\lambda$ has $2$ different $2^{k-1}$-hooks if and only if $\lambda$ has the form $(2^{k-1}-a,2^{k-1}-b+1,2^a,1^{b-a-1})$ for some $0\leq a<b\leq 2^{k-1}-1$. In such a case, $\chi^{\lambda}(1)_2=2$. 
\begin{proof}
Assume that there exist $(i,j)\neq (m,r)$ such that $H^{\lambda}(i,j)=2^{k-1}=H^{\lambda}(m,r) $. It is easy to see that $i \not =m$ and $j\not=r$. Assume that $i<m$. Then we have that $j>r$. Thus, there are $2^{k}-1$ boxes  lying under or at the right of the elements $(i,j)$ and $(m,r)$. This forces $(i,j)=(1,2)$ and $(m,r)=(2,1)$  and then $\lambda $ has the described form. The converse assertion is trivial.

Now, we prove that  $\chi^{\lambda}(1)_2=2$. Let $T_C(\lambda)=\{\lambda_{ij}\mid 0\leq i\leq k, 0\leq j \leq p^i-1\}$ (with the prime $p=2$). Since $\lambda$ has $2$ different $2^{k-1}$-hooks, we have that $b_{k-1}=2$ and $b_{i}=0$ for $i \not=k-1$. Since $n=2^k$, we have that
$$\nu_2(\chi^{\lambda}(1))= \frac{1}{p-1}(\sum_{k\geq 0}b_k-\sum_{k\geq 0}a_k)=2-1=1$$
and the result follows.
\end{proof}
\end{pro}

In the last two results of this subsection we will have to build a new partition by adding a $2^k$-hook to a previous partition. Let us introduce some more notation. Given $k \geq 1$,  $0\leq a \leq k$ and $\lambda$ a partition of a number $n$, we write $A_{k,a}(\lambda)$ to denote the number be the number of $k$-hooks of height $a$ that can be added to $\lambda$ to obtain a partition of $n+k$. We also write $R_{k,a}(\lambda)$ to denote the number be the number of $k$-hooks of height $a$ that can be  removed from $\lambda$ to obtain a partition of $n-k$. The following result is the main result of \cite{B}.

\begin{thm}[Bessenrodt]\label{B}
Let $k \geq 1$ and $0\leq a \leq k$. If $\lambda$ is a partition of a number $n$, then 
$$A_{k,a}(\lambda)=1+R_{k,a}(\lambda).$$
\end{thm}

\begin{lem}\label{ext2}
Let $n$ be even and let $k$ such that $2^{k-1}<n<2^k$. Then for each partition  $\mu$ of $n$ and  each $2^k$-hook $\tau$, there exists a unique partition $\lambda=\lambda(\tau,\mu)$ of $n+2^{k}$ such that $\lambda$ contains $\tau$ as a $2^{k}$-hook and $\lambda \setminus \tau=\mu$. Moreover, $\chi^{\lambda}(1)_2=\chi^{\mu}(1)_2$. 
\begin{proof}
Let $a=h(\tau)$. Since $\mu$ is a partition of $n$ and $n<2^k$, we deduce that it is not possible to remove any $2^k$-cycle from $\mu$. Thus, $R_{k,a}(\mu)=0$ and hence $A_{k,a}=1$ by Theorem \ref{B}. Therefore, there exists a unique $2^k$-hook of height $a$ that can be added to $\mu$ to obtain a partition $\lambda$ of $n+2^k$. Now, we notice that $\lambda$ contains $\tau$ as a $2^k$-hook and $\lambda \setminus \tau=\mu$.

Let us prove that $\chi^{\lambda}(1)_2=\chi^{\mu}(1)_2$. Assume that $T_C(\mu)=\{\mu_{ij}\mid 0\leq i, 0\leq j \leq 2^i-1\}$ and that  $T_C(\lambda)=\{\lambda_{ij}\mid 0 \leq  i, 0\leq j \leq 2^i-1\}$. Since $\mu$ is obtained from $\lambda$ by removing a $2^k$-hook, we deduce that $\mu_{ij}=\lambda_{ij}$ for any $i<k$ and there exists $j_0 \in \{0,\ldots, 2^k-1\}$ such that $\lambda_{kj_0}=(1)$ and $\lambda_{ij}=(0)$ for $j\not =j_0$.

Now, let    $n=\sum_{i= 0 }^{k-1}a_i2^i$ be the $2$-adic expansion of $n$. Then $\sum_{i=0 }^{k-1}a_i2^i+2^k$ is the $2$-adic expansion of $n+2^k$. Let us write $b=\sum_{i\geq0,0\leq j \leq 2^i-1}|\mu_{ij}|$ and $\tilde{b}=\sum_{i\geq0,0\leq j \leq 2^i-1}|\lambda_{ij}|$. It is clear that $\tilde{b}=b+1$ and that $\sum_{i \geq 0}\tilde{a}_i=1+(\sum_{i \geq 0}a_i)$. Then
$$\nu_2(\chi^{\lambda}(1))=\tilde{b}-(\sum_{i = 0}^{k-1}a_i+1)=b-\sum_{i = 0}^{k-1}a_i=\nu_2(\chi^{\mu}(1))$$
and the result follows.
\end{proof}
\end{lem}

\begin{lem}\label{ext}
Let $\mu$ be a partition of $2^{k}$, such that $\mu$ is not a $2^k$-hook. Then for each $2^k$-hook $\tau$, there exists a unique partition $\lambda=\lambda(\tau,\mu)$ of $2^{k+1}$ such that $\lambda$ contains $\tau$ as a $2^{k}$-hook and $\lambda \setminus \tau=\mu$. Moreover, $\chi^{\lambda}(1)_2=2\cdot \chi^{\mu}(1)_2$.
\begin{proof}
Let $a=h(\tau)$. Since $\mu$ is not a $2^k$-hook, we deduce that it is not possible to remove any $2^k$-cycle from $\mu$. Thus, $R_{k,a}(\mu)=0$ and hence $A_{k,a}=1$ by Theorem \ref{B}. Therefore, there exists a unique $2^k$-hook of height $a$ that can be added to $\mu$ to obtain a partition $\lambda$ of $2^{k+1}$. Now, we notice that $\lambda$ contains $\tau$ as a $2^k$-hook and $\lambda \setminus \tau=\mu$.

Let us prove that $\chi^{\lambda}(1)_2=2\cdot\chi^{\mu}(1)_2$. Assume that $T_C(\mu)=\{\mu_{ij}\mid0 \leq i, 0\leq j \leq 2^i-1\}$ and that  $T_C(\lambda)=\{\lambda_{ij}\mid i \geq 0, 0\leq j \leq 2^i-1\}$. Since $\mu$ is obtained from $\lambda$ by removing a $2^k$-hook, we deduce that $\mu_{ij}=\lambda_{ij}$ for any $i<k$ and there exists $j_0 \in \{0,\ldots, 2^k-1\}$ such that $\lambda_{kj_0}=(1)$ and $\lambda_{ij}=(0)$ for $j\not =j_0$.

Now, let us write   $b=\sum_{i\geq0,0\leq j \leq 2^i-1}|\mu_{ij}|$ and $\tilde{b}=\sum_{i\geq0,0\leq j \leq 2^i-1}|\lambda_{ij}|$. It is clear that $\tilde{b}=b+1$. Then
$$\nu_2(\chi^{\lambda}(1))=\tilde{b}-1=(b-1)+1=\nu_2(\chi^{\mu}(1))+1$$
and the result follows.
\end{proof}
\end{lem}

\subsection{Sylow subgroups of $\mathsf{S}_n$ and their normalizers}\label{SubsectCentralizers}

In this subsection we describe the structure of the normalizers of the Sylow $p$-subgroups of $\mathsf{S}_n$. See Section 2 of \cite{G} for a proof of the results presented in this subsection. Let $p$ be a prime and let $k \geq 1$ be an integer. Assume that $P_{p^k}\in \Syl_{p}(\mathsf{S}_{p^k})$. Then 
$$P_{p^k}\cong \mathsf{C}_p\wr \ldots \wr \mathsf{C}_p,$$ 
for $k$ wreath factors. Moreover, $$\mathbf{N}_{\mathsf{S}_{p^k}}(P_{p^k})\cong P_{p^k}\rtimes (\mathsf{C}_{p-1})^k.$$

Now, let $n\geq 1$ be any integer. Assume that $n=\sum_{i=1}^{f}a_ip^{k_i}$ with $0\leq k_1<k_2\ldots <k_f$ and $0<a_i<p$ for all $i$ is the $p$-adic expansion of $n$. If $P_n\in \Syl_p(\mathsf{S}_n)$, then 
$$P_n\cong P_{p^{k_1}}^{a_1}\times \ldots \times P_{p^{k_f}}^{a_f}$$
 and 
$$\mathbf{N}_{\mathsf{S}_n}(P_n)\cong (\mathbf{N}_{\mathsf{S}_{p^{k_1}}}(P_{p^{k_1}})\wr \mathsf{S}_{a_1})\times \ldots \times(\mathbf{N}_{\mathsf{S}_{p^{k_f}}}(P_{p^{k_f}})\wr \mathsf{S}_{a_f}).$$

In particular, for $p=2$ we have that $\mathbf{N}_{\mathsf{S}_n}(P_n)=P_n$. Moreover, we have 
$$|\Irr_{2'}(\mathbf{N}_{\mathsf{S}_n}(P_n))|=|\Irr_{2'}(P_n)|=|\Lin(P_n)|=2^{\sum_{i=1}^f k_i}.$$

Since we are dealing with wreath products, it is convenient to recall the following result.

\begin{thm}\label{Mattarei}
Let $H$ be a group, let $A\leq \Sym(n)$ and let $G=H\wr A$. Let  $\theta=\theta_1\times\ldots\times\theta_n\in \Irr(H^n)$. For $\phi \in \Irr(H)$ we write $\theta_{\phi}=\{i\in \{1,\ldots,n\}\mid\theta_i=\phi\}$. Then
$$\I_{G}(\theta)=\prod_{\phi\in \Irr(H)}H\wr \Stab_A(\theta_{\phi})$$
where we embed $\Stab_A(\theta_{\phi})=\{a \in A\mid a(\theta_{\phi})=\theta_{\phi}\}$ into $ \Sym(\theta_{\phi})$ in the natural way. Moreover, $\theta$ extends to an irreducible character of $\I_{G}(\theta)$.
\begin{proof}
The first part can be deduced by straightforward calculation. The second part is Lemma 1.3 of  \cite{Mattarei}.
\end{proof}
\end{thm}

\subsection{Picky elements in the symmetric group}

The following results appear in \cite{MMM} and \cite{MRPicky}, respectively. For completeness, we include their short proofs.

\begin{lem}\label{Centralizer}
Let $G$ be a group, let $p$ be a prime and let $P\in \Syl_p(G)$. If $x\in P$ is  picky, then $\mathbf{C}_G(x)\subseteq \mathbf{N}_G(P)$. In particular, $\mathbf{C}_G(x)=\mathbf{C}_{\mathbf{N}_G(P)}(x)$.
\begin{proof}
Let $g\in \mathbf{C}_G(x)$. We have that $x=x^g$ and hence $x\in P\cap P^g$. Since $x$ is picky, this forces $P=P^g$, or equivalently, $g \in \mathbf{N}_G(P)$.
\end{proof}
\end{lem}

\begin{lem}\label{Facil}
Let $G$ be a group, let $p$ be a prime and let $P\in \Syl_p(G)$. Assume that $x, y\in P$ are picky. If $x$ and $y$ are $G$-conjugate, then $x$ and $y$ are $\mathbf{N}_G(P)$-conjugate.
\begin{proof}
Assume that $y=x^g$ for some $g \in G$. Then $y\in P$ and $y=x^g \in P^g$, which forces $y \in P\cap P^g$. Since $y$ is picky, we deduce that $P=P^g$. 
\end{proof}
\end{lem}

Now, we work towards classifying  picky elements in symmetric groups.

\begin{lem}\label{GunterPrevio}
 Let $G = H_1 \times \cdots\times H_r$. Then $g \in G$ is picky if and only if the projection
of $g$ into each component $H_i$ is picky.
\begin{proof}
Let $g = (g_1, \ldots , g_r) \in G$, where $g_i\in  H_i$. If $g_i$ lies in a Sylow $p$-subgroup $P_i$ of $H_i$,
then $g$ lies in the Sylow $p$-subgroup $P_1\times\cdots \times P_r$ of $G$. Conversely, if $g \in P$, with $P$ a Sylow $p$-subgroup of $G$, then $g_i\in P_i:=P\cap H_i$, a Sylow $p$-subgroup of $H_i$ with $P=P_1\times\cdots \times P_r$. The result follows.
\end{proof}
\end{lem}

Let $p$ be a prime and let $n$ be a positive integer. Assume that $n=\sum_{i=0}^fa_ip^i$ is the $p$-adic expansion of $n$. An element $x\in \mathsf{S}_n$ is called $p$-adic if it contains $a_i$ cycles of length $p^i$ for every $0\leq i \leq f$. The following result was proved by Malle for $p>3$ and is included here with his kind permission.

\begin{pro}\label{ClassificationPickySymmetric}
Let $p$ be a prime and let  $n\geq 1$ be an integer. Then $x\in \mathsf{S}_n$ is a picky $p$-element if and only if one of the following holds
\begin{itemize}
\item [\underline{Type I:}] $x$ is a $p$-adic element of $n$.

\item[\underline{Type II:}] $p=2$,  $n\geq 6$ is even and $x$ is a $2$-adic element of $n-2$ plus two fixed points.

\item[\underline{Type III:}] $p=3$, $n \equiv \pm3\pmod{9}$,  $x$ is a $3$-adic element of $n-3$ plus three fixed points
\end{itemize}
\begin{proof}
For $p=2$, the result follows by Theorem 4.5 of \cite{MMM}. For the remainder $p$ will be an odd prime.

We notice that if $x$ is a $p$-adic element of $\mathsf{S}_n$, then $x$ is a picky element of $\mathsf{S}_n$ by Theorem 4.2 of \cite{MMM}.

Now, we prove that any element of Type III is picky. Let $n=3a+ \sum_{i=2}^fa_i3^i$ for $a\in \{1,2\}$ and $a_i\{0,1,2\}$ for every $2\leq i \leq 3$ and  let $x$ be a picky element of Type III. Let $m=\sum_{i=2}^fa_i3^i$ and let $k=3a$. Let $H_1=\mathsf{S}_m\times \mathsf{S}_k\leq \mathsf{S}_n$ with $x \in H_1$, let $x_1$ be the projection on $\mathsf{S}_m$ and let $x_2$ be the projection of $x$ on $\mathsf{S}_k$. Let $P\in \Syl_3(\mathsf{S}_n)$ with $x\in P$. Since $x$ contains $a_i$ cycles of length $3^i$ for every $i \geq 2$, then $P=P_1\times P_2$ for  $P_1 \in \Syl_3(\mathsf{S}_m)$ with $x_1 \in P_1$ and for $P_2 \in \Syl_3(\mathsf{S}_k)$ with $x_2 \in P_2$. We observe that $x_1$ is a $3$-adic element of $\mathsf{S}_m$ and hence it is picky in $\mathsf{S}_m$. On the other hand, we have that $(k, x_2)\in\{(3,\Id),(6,(1,2,3))\}$ and hence, an easy inspection proves that $x_2$ is picky in $\mathsf{S}_k$. Since $P_1$ and $P_2$ are the unique Sylow $3$-subgroups containing $x_1$ and $x_2$, respectively, then $P=P_1\times P_2$ is the unique Sylow $3$-subgroup of $\mathsf{S}_n$ containing $x$. Thus, $x$ is picky in $\mathsf{S}_n$.

Thus, the elements of Types I, II and III are picky. Now, we prove that these are all the picky elements. We divide the proof  in a series of steps. Let $x \in \mathsf{S}_n$.

\underline{Case $n=p$:}  In this case $x$ is either a $p$-cycle or $x$ is the identity. If $x$ is a $p$-cycle, then $x$ is a $p$-adic element of $\mathsf{S}_p$ and hence it is picky.  If $x$ is the identity, then $x$ lies in every Sylow $p$-subgroup of $\mathsf{S}_p$ and hence, it is picky if and only if $\mathsf{S}_p$ possesses a unique Sylow $p$-subgroup, or equivalently, when $p=3$. In this case $x=\Id$ is a picky element of Type III in $\mathsf{S}_3$.

\underline{Case $n=p^i$ for $i\geq 2$:}  The natural subgroup $H=\mathsf{S}_{p^{i-1}}\wr\mathsf{S}_p$ of $\mathsf{S}_{p^i}$ contains a Sylow $p$-subgroup of $\mathsf{S}_{p^i}$. Now, we notice that the base group $N= \mathsf{S}_{p^{i-1}}^p$ contains elements of all cycle types consisting solely of cycles of $p$-power order
less than $p^i$. If $p>3$, then $|\Syl_p(H/N)|>1$ and hence  such elements are not picky as they lie in the preimages of the various
Sylow $p$-subgroups of $H/N = \mathsf{S}_p$. Thus, if $p>3$ and $x$ is picky in $\mathsf{S}_n$, then $x$ is a $p^i$-cycle.

 Now, let us assume that $n=3^i$ for $i \geq 2$. For $n=9$, using GAP \cite{gap}, we see that the unique picky $3$-elements are the $9$-cycles. Let us assume that $i\geq 3$ and that the unique picky $3$-elements of $\mathsf{S}_{3^{i-1}}$ are the $3^{i-1}$-cycles. Let us assume that $x\in \mathsf{S}_{3^i}$ is a picky $3$-element, which is not a $3^i$-cycle. Then $x$ is contained in a group $N$ as above. Moreover, the projections of $x$ in each copy of $\mathsf{S}_{3^{i-1}}$ is picky. Thus, $x$ is a product of $3$ disjoint $3^{i-1}$-cycles.

Let $\Omega_1=\{1,2,\ldots, 3^{i-1}\}$, $\Omega_2=\Omega_1+3^{i-1}$ and $\Omega_3=\Omega_1+2\cdot 3^{i-1}$ and let $P_i \in \Syl_3(\mathsf{S}_{\Omega_i})$. Let  $$x=(1,4,\ldots, 3^i-2)(2,5,\ldots,3^i-1)(3,6,\ldots,3^i).$$ We know that $P=(P_1\times P_2\times P_3)\rtimes \langle x\rangle\in \Syl_3(\mathsf{S}_{n})$ and $x \in P$. Let us consider the element $$y=(1,4,\ldots, 3^i-2).$$ Then $y \in \mathbf{C}_{\mathsf{S}_n}(x)$ but $y \not\in \mathbf{N}_{\mathsf{S}_n}(P)$. Thus, $x$ cannot be picky by Lemma \ref{Centralizer}.

\underline{General case:} Let us assume that $n =\sum_{i=0}^fa_ip^i$ is the $p$-adic expansion of $n$. Then the subgroup $H=\prod_{i=0}^f\mathsf{S}_{p^i}^{a_i}$ of $G=\mathsf{S}_n$ contains  a Sylow $p$-subgroup of $G$. Thus, if $x \in H$ is a picky $p$-element, then it is so in $H$. It now follows from the previous case in conjunction with Lemma \ref{GunterPrevio}  that if $x$ is picky, then the projection of $x$ in each $\mathsf{S}_{p^i}$ is either a $p^i$-cycle or $p^i=3$. Thus, if $p>3$, then $x$ is a $p$-adic element.

Now, let us assume that $p=3$. Let $m=\sum_{i=2}^fa_i3^i$ and let $k=3a_1+a_0\leq 8$. Let $H_1=\mathsf{S}_m\times \mathsf{S}_k\leq \mathsf{S}_n$, let $x_1$ be the projection on $\mathsf{S}_m$ and let $x_2$ be the projection of $x$ on $\mathsf{S}_k$. We notice that if $x$ is picky in $\mathsf{S}_n$, then $x_1$ is picky in $\mathsf{S}_m$ and $x_2$ is picky in $\mathsf{S}_k$.  By the previous reasoning, we have that $x_2$ is a $3$-adic element. Moreover, inspecting the picky $3$-elements of $\mathsf{S}_m$ for $m \leq 8$, we deduce that $x_1$ is either $3$-adic in $\mathsf{S}_m$  or $(m,x_1)\in\{(3,\Id),(6,(1,2,3))\}$.  In the first case, $x$ is a $3$-adic element of $\mathsf{S}_n$ and in the second case $x$ is a picky element of Type III. The result follows.
\end{proof}
\end{pro}

\begin{lem}\label{Integers}
Let $p$ be a prime, let  $n$ be an integer and $x\in \mathsf{S}_n$ be a picky $p$-element. Assume that $P_n \in \Syl_p(\mathsf{S}_n)$ is the unique Sylow $p$-subgroup containing $x$. If $H \in \{\mathsf{S}_n,\mathbf{N}_{\mathsf{S}_n}(P_n)\}$, then $\chi(x)\in \Z$ for any $\chi \in \Irr(H)$.
\begin{proof}
We know that $x$ is $\mathsf{S}_n$-conjugate to $x^j$ for any $j$ with $(j,o(x))=1$.  Since $x$ is picky, $x$ is $\mathbf{N}_{\mathsf{S}_n}(P_n)$-conjugate to $x^j$ for any $j$ with $(j,o(x))=1$, by Lemma \ref{Facil}. The result now follows.
\end{proof}
\end{lem}

\begin{pro}\label{GLLV}
Let $p$ be a prime, let  $n$ be an integer and $x\in \mathsf{S}_n$ be a $p$-adic element of $\mathsf{S}_n$. Then $\Irr^x(\mathsf{S}_n)=\Irr_{p'}(\mathsf{S}_n).$
\begin{proof}
This was proved during the proof of Theorem 3.16 of \cite{GLLV}.
\end{proof}
\end{pro}

We begin by studying the case when $x$ is a $2$-adic element.

\begin{lem}\label{TypeI}
Let $x \in \mathsf{S}_n$ be a $2$-adic element and let $P_n\in \Syl_2(\mathsf{S}_n)$ be the unique Sylow $2$-subgroup of $\mathsf{S}_n$ containing $x$. Then $\Irr^x(\mathsf{S}_n)= \Irr_{2'}(\mathsf{S}_n)$ and  $\Irr^x(P_n)= \Irr_{2'}(P_n)$. Moreover, $\chi(x),\psi(x)\in \{\pm1\}$ for every $\chi \in \Irr^x(\mathsf{S}_n)$ and every $\psi \in \Irr^x(P_n)$. 
\begin{proof}
Let $n=\sum_{i=1}^{f}2^{n_i}$ with $0\leq n_1<n_2<\ldots<n_f$. By straightforward calculations, we deduce that $|\mathbf{C}_{\mathsf{S}_n}(x)|=|\mathbf{C}_{P_n}(x)|=2^{\sum_{i=1}^f n_i}$ (we recall that $\mathbf{C}_{\mathsf{S}_n}(x)=\mathbf{C}_{P_n}(x)$ by Lemma \ref{Centralizer}).

 Let $H\in \{\mathsf{S}_n, P_n\}$. By Theorem \ref{2PrimeDegree} and  the comments in subsection \ref{SubsectCentralizers}, we have that $|\Irr_{2'}(H)|=2^{\sum_{i=1}^f n_i}$. Moreover, by Lemma \ref{Integers},  $\chi(x)\in \Z\setminus \{0\}$ for any $\chi \in \Irr_{2'}(H)$. Thus, applying the Second Orthogonality Relation, we deduce that 
$$2^{\sum_{i=1}^f n_i}=|\mathbf{C}_{H}(x)|=\sum_{\chi\in \Irr^x(H)}|\chi(x)|^2\geq \sum_{\chi\in \Irr_{2'}(H)}|\chi(x)|^2\geq |\Irr_{2'}(H)|=2^{\sum_{i=1}^f n_i}.$$
Thus, $|\Irr^x(H)|=|\Irr_{2'}(H)|$ and $|\chi(x)|=1$ for any $\chi \in \Irr_{2'}(H)$. Since $\chi(x)\in \Z$, we have that $\chi(x)=\pm1$.
\end{proof}
\end{lem}

Now, we move to study the elements of Type II. Let $n\geq 2$ be even and  $y\in \mathsf{S}_n$ be a picky element of Type II. Then $|\mathbf{C}_{\mathsf{S}_n}(y)|=2^{1+\sum_{i=1}^t m_i}$, where $n-2= \sum_{i=1}^t 2^{m_i}$ is the $2$-adic expansion of $n-2$. The following elementary  lemma determines when $|\mathbf{C}_{\mathsf{S}_n}(y)|=|\Irr_{2'}(\mathsf{S}_n)|$.

\begin{lem}\label{Sizes}
Let $n\geq 4$ be an even integer. Let $n=\sum_{i=1}^f 2^{n_i}$ and $n-2=\sum_{i=1}^t 2^{m_i}$ be the $2$-adic expansions of $n$ and $n-2$, respectively. Then $\sum_{i=1}^t m_i\geq \sum_{i=1}^f n_i$ with equality if and only if $n \not \equiv 0 \pmod{8}$.
\begin{proof}
The proof follows by straightforward calculation.
\end{proof}
\end{lem}

\begin{lem}\label{8notdivides}
Let $n $ be an even integer with $n \not \equiv 0 \pmod{8}$. Let $y\in \mathsf{S}_n$  be a picky element of Type II and let $P_n\in \Syl_2(\mathsf{S}_n)$ be the unique Sylow $2$-subgroup of $\mathsf{S}_n$ containing $y$. Then $\Irr^y(\mathsf{S}_n)= \Irr_{2'}(\mathsf{S}_n)$ and  $\Irr^y(P_n)= \Irr_{2'}(P_n)$. Moreover, $\chi(y)=\pm1$ for every $\chi \in \Irr_{2'}(\mathsf{S}_n)$ and every $\chi \in \Irr_{2'}(P_n)$. 
\begin{proof}
Let $n=\sum_{i=1}^f 2^{n_i}$ and $n-2=\sum_{i=1}^t 2^{m_i}$ be the $2$-adic expansions of $n$ and $n-2$, respectively. Let $H$ be $\mathsf{S}_n$ or $P_n$.  By Lemma \ref{Sizes}, we have that 
$$|\mathbf{C}_{H}(y)|=2^{1+\sum_{i=1}^t m_i}=2^{\sum_{i=1}^f n_i}=|\Irr_{2'}(H)|.$$ Now, the result follows by reasoning as in the proof of Theorem \ref{TypeI}.
\end{proof}
\end{lem}

Now, we prove that Theorem \ref{thmA} holds for $\mathsf{S}_n$ provided that $8$ does not divide $n$.

\begin{cor}
Let $n $ be an even integer with $n \not \equiv 0 \pmod{8}$. Then Theorem \ref{thmA} holds for $\mathsf{S}_n$.
\begin{proof}
Let  $P_n\in \Syl_2(\mathsf{S}_n)$ and let $x,y \in P_n$ be picky elements of types I and II respectively. From Lemmas \ref{TypeI} and \ref{8notdivides}, we deduce that $\Irr^x(\mathsf{S}_n)=\Irr^y(\mathsf{S}_n)= \Irr_{2'}(\mathsf{S}_n)$ and that $\Irr^x(P_n)=\Irr^y(P_n)= \Irr_{2'}(P_n)$. Therefore, $\Irr^{\mathcal{P}}(\mathsf{S}_n)= \Irr_{2'}(\mathsf{S}_n)$ and $\Irr^{\mathcal{P}}(P_n)= \Irr_{2'}(P_n)$. 

Let $\Gamma: \Irr_{2'}(\mathsf{S}_n)\rightarrow \Irr_{2'}(P_n)$ be a bijection. Then $\Gamma$ satisfies properties (I) and (III) trivially. Moreover, from  Lemmas \ref{TypeI} and \ref{8notdivides}, we deduce that $\Gamma(\chi)(g),\chi(g) \in \{\pm1\}$ for every $\chi \in  \Irr_{2'}(\mathsf{S}_n)$ and every $g \in \mathcal{P}$. Thus, $\Gamma$ satisfies also property (II) and the result follows.
\end{proof}
\end{cor}

Now, we study the case when $x\in \mathsf{S}_{p^k}$ is a $p^k$-cycle for a prime $p$.

\begin{lem}\label{Casepk}
Let $p$ be a prime,  let $k \geq 1$ be an integer, and  let $x \in \mathsf{S}_{p^k}$ be a $p^k$-cycle. Assume that $P_{p^k}\in \Syl_p(\mathsf{S}_{p^k})$ is the unique Sylow $p$-subgroup of $\mathsf{S}_{p^k}$ containing $x$.  Then $\Irr^x(\mathsf{S}_{p^k})= \Irr_{p'}(\mathsf{S}_{p^k})$ and  $\Irr^x( \mathbf{N}_{\mathsf{S}_{p^k}}(P_{p^k}))= \Irr_{p'}( \mathbf{N}_{\mathsf{S}_{p^k}}(P_{p^k}))$. Moreover, $\chi(x),\psi(x)\in \{\pm1\}$ for every $\chi \in \Irr^x(\mathsf{S}_{p^k})$ and every $\psi \in \Irr^x( \mathbf{N}_{\mathsf{S}_{p^k}}(P_{p^k}))$. 
\begin{proof}
Let  $H \in \{\mathsf{S}_{p^k}, \mathbf{N}_{\mathsf{S}_{p^k}}(P_{p^k})\}$.  From the above, we know that $|\mathbf{C}_{H}(x)|=p^k=|\Irr_{p'}(H)|$ and that $\chi(x) \in \Z$ for any $\chi \in \Irr_{p'}(H)$. Thus, we have that
$$p^k=|\mathbf{C}_{H}(x)|\geq \sum_{\chi\in \Irr_{p'}(H)}|\chi(x)|^2\geq |\Irr_{p'}(H)|=p^k$$
which implies that $\Irr^x(H)= \Irr_{p'}(H)$ and that $\chi(x)= \pm1$ for any $\chi \in\Irr_{p'}(H)$.
\end{proof}
\end{lem}

\section{The case $p=2$ of the Subnormalizer Conjecture} \label{SubnormalPart}

In this section we prove Theorem \ref{thmB}.  We begin by introducing some results on the subnormalizers. For the remainder, given a group $G$ and a prime $p$, we will write $G_p$ to denote the set of $p$-elements of $G$.

\begin{pro}[Proposition 2.6 of \cite{Malle}]\label{MalleSub}
Let $p$ be a prime and let $G$ be a group. Given $x\in G_p$, we have 
$$\Sub_G(x)=\langle \mathbf{N}_G(P)\mid P\in \Syl_p(G),x\in P\rangle.$$
\end{pro}

From this result we can deduce the following corollary. As we pointed out in the introduction, this result can also be deduced from the  results in \cite{Casolo1}.

\begin{cor}\label{PickySubnormal}
Let $p$ be a prime and let $G$ be a group. Assume that $P \in \Syl_p(G)$ and $x \in P$. Then $x$ is picky if and only if $\Sub_G(x)=\mathbf{N}_G(P)$.
\end{cor}

\begin{lem}\label{SubnormalProduct}
Let $G=H\times K$,  let $x\in H$, and $y \in K$. Then 
$$\Sub_{H}(x)\times \Sub_{K}(y)\leq \Sub_{G}((x,y)).$$
\begin{proof}
Let $z\in H$ such that there exists a chain 
$$\langle x \rangle \triangleleft N_1 \triangleleft \ldots \triangleleft N_t \triangleleft \langle x,z\rangle.$$
Then, we have that
$$\langle (x,y) \rangle  \triangleleft N_1\times \langle y \rangle \triangleleft \ldots \triangleleft N_t\times \langle y \rangle \triangleleft \langle x,z\rangle \times \langle y \rangle.$$
Thus $\langle (x,y) \rangle  \trianglelefteq \trianglelefteq  \langle x,z\rangle \times \langle y \rangle $ and since $\langle (x,y) \rangle\leq  \langle (x,y),(z,1)\rangle \leq \langle x,z\rangle \times \langle y \rangle $ we deduce that $\langle (x,y) \rangle  \trianglelefteq \trianglelefteq  \langle (x,y),(z,1)\rangle.$ Therefore, $(z,1)\in \Sub_{G}((x,y))$. It follows that $\Sub_{H}(x)\times1\leq \Sub_{G}((x,y))$. 

Analogously, we prove that $1\times\Sub_{K}(y)\leq \Sub_{G}((x,y))$ and the result follows.
\end{proof}
\end{lem}

 To calculate the subnormalizers of $p$-elements in symmetric groups we need to introduce an alternative way to see the Sylow subgroups of $\mathsf{S}_n$. Let $n=\sum_{i=0}^{k}a_{i}p^{i}$ be the $p$-adic expansion of $n$. We make $a_k$ disjoint subsets of size $p^k$ in $\{1,\ldots,n\}$. Inside each subset of size $p^k$ we make $p$ disjoint subsets of size $p^{k-1}$. We repeat this process till we get sets of size $1$. With the numbers not lying in the sets of size $p^k$ we make $a_{k-1}$ disjoint subsets of size $p^{k-1}$ and inside each of them we make subsets of size $p^i$ for $i\leq k-2$ as before. We repeat this process for each $j=k,\ldots, 0$. Taking all subsets obtained by the previous process we obtain a block structure for $n$. Now, given a block structure $B$ and $\sigma\in \mathsf{S}_n$ we will say that $\sigma$ preserves the structure $B$ if $\sigma(b)\in B$ for every $b \in B$ and we will write that $\sigma(B)=B$. We have the following result, which follows from the results in \cite{f}.

\begin{thm}\label{SylowBlock}
	Let $n$ be an integer and let $B$ be a block structure of $n$. If we set $P=\{\sigma \in (\mathsf{S}_n)_p\mid\sigma(B)=B\}$, then $P\in \Syl_p(\mathsf{S}_n)$. In addition, each Sylow $p$-subgroup of $\mathsf{S}_n$ can be associated with a unique block structure.
\end{thm}

Before continuing, we have to introduce more notation.  Given $x \in \mathsf{S}_{n}$, we  define the support of $x$ by
$$\Sup(x)=\{i \in \{1,2,\ldots,n\}\mid x(i)\neq i\}.$$
Given $x\in \mathsf{S}_{n}$ we know that $x$ can be expressed as the product of disjoint cycles and hence, we can associate a partition of $n$ to $x$. We define the type of $x$ as the multiset of lengths of the disjoint cycles of  $x$.  In addition, we write $\fix(x)$ to denote the number of fixed points of $x$ in the natural action of $\mathsf{S}_n$ on $\{1,\ldots,n\}$. We have the following result.

\begin{lem}\label{NoFixedPoint}
Let $k \geq 3$ and let $x \in \mathsf{S}_{2^k}$ be  a $2$-element. If $\fix(x)=0$, then there exists  a block structure $B$ possessing two subsets $b_1$ and $b_2$ such that $|b_1|=|b_2|=2^{k-1}$ and $x(b_1)=b_2$.
\begin{proof}
We prove the result by induction on $k$. For $k=3$ the result follows from straightforward calculation. Let us assume that the result holds for $k-1$.

Assume first that $x$ is a $2^k$-cycle.  Without loss of generality we may assume that $x=(1,2\ldots, 2^k)$.  Now, let us define a block structure $B$ on $\{1,2,\ldots, 2^k\}$ as follows. For each $j \in \{0,\ldots, k\}$, $B$ contains $2^j$ blocks $b_{j1},\ldots, b_{j2^j}$ of size $2^{k-j}$ and two elements $a,c\in \{1,2,\ldots, 2^k\}$ lie in the same block if and only if $a\equiv c \pmod 2^{k-j}$. It is easy to see that $x(B)=B$ and $x(b_1)=b_2$, where $b_1$ and $b_2$ are the two sets of size $2^{k-1}$ in $B$.

Assume now that $x$ is not  a $2^k$-cycle. Then $x=yz$ with $\Sup(y)\cap \Sup(z)=\varnothing$. Without loss of generality, we may assume that $\Sup(y)=\{1,2\ldots, 2^{k-1}\}$ and $\Sup(z)=\{2^{k-1}+1,\ldots,2^k\}$. By inductive hypothesis there exist two block structures $B_1$ and $B_2$ of $\{1,2\ldots, 2^{k-1}\}$ and $\{2^{k-1}+1,\ldots,2^k\}$, respectively, such that $B_i$ possesses two subsets $b_{i1}$ and $b_{i2}$ of size $2^{k-2}$ such that $y(b_{11})=b_{12}$ and $y(b_{21})=b_{22}$. Now, we define $B:=B_1\cup B_2\cup \{b_{11}\cup b_{21}, b_{12}\cup b_{22},\{1,\ldots ,2^k\}  \}$. It is not hard to see that $x(B)=B$ and $x(b_{11}\cup b_{21})=b_{12}\cup b_{22}$.
\end{proof}
\end{lem}

 Given $\Omega \subseteq \{1,\ldots,n\}$, we write $\mathsf{S}_{\Omega}:=\Sym(\Omega)$ and we embed it into $\mathsf{S}_n$. Let $y \in   \mathsf{S}_{2^k}$  such that $\Sup(y)\subseteq \Omega$ for some $\Omega \subseteq \{1,2\ldots,2^k\}$ with  $|\Omega|=2^{k-1}$. Then, we may identify $y$ with an element in $\mathsf{S}_{\Omega}\cong \mathsf{S}_{2^{k-1}}$. In this situation, we will write $\Sub_{\mathsf{S}_{2^{k-1}}}(y)$ to denote the subnormalizer of $y$ in $\mathsf{S}_{\Omega}$. Thus, the subgroup $\Sub_{\mathsf{S}_{2^{k-1}}}(y)\wr \mathsf{S}_2$ of $ \mathsf{S}_{2^{k}}$ is well defined.  In the case when $G=\mathsf{S}_{2^k}$ the following result classifies the $2$-elements of $\mathsf{S}_{2^k}$ with a proper subnormalizer.

\begin{thm}\label{Subnormalizers}
Let $k \geq 3$ be an integer and let $x \in \mathsf{S}_{2^k}$ be a $2$-element. Then $\Sub_{\mathsf{S}_{2^k}}(x)<\mathsf{S}_{2^k}$ if and only if  one of the following holds:
\begin{itemize}
\item[(i)] $x$ is a $2^k$-cycle. 

\item[(ii)] $o(x)=2^{k-1}$ and $\fix(x)>0$. 
\end{itemize}
Moreover, in case (ii),  $\Sub_{\mathsf{S}_{2^k}}(x)=\Sub_{\mathsf{S}_{2^{k-1}}}(y)\wr \mathsf{S}_2$, where $y$ is the product of all cycles of length smaller than $2^{k-1}$ in the expression of $x$ as a product of disjoint cycles.
\begin{proof}
We prove the result by induction on $k$. For $k=3$, the result can be checked by GAP \cite{gap}. By inductive hypothesis, we assume that the result holds for $k-1$.   We study the subnormalizers of elements in different cases.

\underline{Case $o(x)=2^k$:} In this case  $x$ is a $2^k$-cycle. Thus, $x$ is a picky element and hence $\Sub_{\mathsf{S}_n}(x)=P<\mathsf{S}_n$, where $P$ is the unique Sylow $2$-subgroup of $\mathsf{S}_n$ containing $x$.

\underline{Case $o(x)=2^{k-1}$:} In this case, $x=zy$, where $z$ is a $2^{k-1}$-cycle and $y \in \mathsf{S}_{2^{k-1}}$. Without loss of generality, we may assume  that $z=(1, \ldots,2^{k-1} )$.  Let $B_1$ be a block structure on $\{1,\ldots, 2^{k-1}\}$ such that $z(B_1)=B_1$ and let $B_2$ be a block structure on $\{2^{k-1}+1,\ldots, 2^{k}\}$ such that $y(B_2)=B_2$. Then $B=B_1\cup B_2\cup \{1,\ldots, 2^k\}$ is a block structure on  $\{1,\ldots, 2^{k}\}$ with $x(B)=B$. It follows that if  $Q\in \Syl_2(\mathsf{S}_{2^{k-1}})$ satisfies  $y \in Q$, then  $ Q\wr \mathsf{S}_2\leq \Sub_{\mathsf{S}_{2^k}}(x)$. It follows that $\Sub_{\mathsf{S}_{2^{k-1}}}(y)\wr \mathsf{S}_2\leq \Sub_{\mathsf{S}_{2^k}}(x)$.

\begin{itemize}

\item \underline{Case $\fix(x)=\fix(y)>0$:}  Let $B$ be a block structure on $\{1,\ldots, 2^k\}$ such that $x(B)=B$. Let $a\in \{1,\ldots, 2^k\}$ such that $x(a)=a$. Then $a>2^{k-1}$. Now, let $b \in B$ such that $a\in b$ and $|b|=2^{k-1}$. Since $x(a)=a$, we deduce that $x(b)=b$. We claim that $b=\{2^{k-1}+1,\ldots,2^k\}$. Seeking for a contradiction, let us assume that there exists $i \in b\cap \{1,\ldots , 2^{k-1}\}$. Then $x^{j}(i)\in b$ for every $j\geq 0$. Therefore $a\cup \{1,\ldots , 2^{k-1}\}\subseteq b$, which is a contradiction since $|b|=2^{k-1}$ and $a>2^{k-1}$. The claim follows. Thus, $x(B)=B$ if and only if $B=B_1\cup B_2\cup \{1,\ldots, 2^k\}$, where $B_1$ is a block structure on $\{1,\ldots, 2^{k-1}\}$ such that $z(B_1)=B_1$ and $B_2$ is a block structure on $\{2^{k-1}+1,\ldots, 2^{k}\}$ such that $y(B_2)=B_2$.  We deduce that $\Sub_{\mathsf{S}_{2^{k-1}}}(y)\wr \mathsf{S}_2= \Sub_{\mathsf{S}_{2^k}}(x)$.

\item \underline{Case $\fix(x)=\fix(y)=0$:} We study the case when $o(y)=2^{k-1}$ and the case $o(y)<2^{k-1}$ separately.

\begin{itemize}

\item [a)] \underline{Case $o(y)<2^{k-1}$:} In this case, we have that $\Sub_{\mathsf{S}_{2^{k-1}}}(y)=\mathsf{S}_{2^{k-1}}$ by inductive hypothesis. It follows that $\mathsf{S}_{2^{k-1}}\wr \mathsf{S}_2\leq \Sub_{\mathsf{S}_{2^k}}(x)$ and hence, it suffices to find an element in $\Sub_{\mathsf{S}_{2^k}}(x)$ not lying in  $\mathsf{S}_{2^{k-1}}\wr \mathsf{S}_2$. Therefore, it suffices to find a block structure $B$ of $\{1,\ldots,2^k\}$ such that $x(B)=B$ and $\{1,2\ldots, 2^{k-1}\} \not\in B$. By Lemma \ref{NoFixedPoint}, there exist block structures $B_1$ and $B_2$ of $\{1,2\ldots, 2^{k-1}\}$ and $\{2^{k-1}+1,\ldots,2^k\}$, respectively, such that $B_i$ possesses two subsets $b_{i1}$ and $b_{i2}$ of size $2^{k-2}$ such that $z(b_{11})=b_{12}$ and $y(b_{21})=b_{22}$. Now, we define $B:=B_1\cup B_2\cup \{b_{11}\cup b_{21}, b_{12}\cup b_{22},\{1,\ldots ,2^k\}  \}$. It is not hard to see that $x(B)=B$. Thus, $\Sub_{\mathsf{S}_{2^k}}(x)=\mathsf{S}_{2^k}$.

\item [b)]  \underline{Case $o(y)=2^{k-1}$:} In this case,  $x$ is the product of $2$ disjoint $2^{k-1}$-cycles. Without loss of generality, we may assume that 
$$x=(1,2,\ldots,2^{k-1})(2^{k-1}+1,\ldots, 2^k).$$
We claim that the elements 
$$g=(1,2^{k-1}+1,2,2^{k-1}+2,\ldots, 2^{k-1},2^k)$$
and
$$h=(1,2^{k-1}+1)$$
lie in $\Sub_{\mathsf{S}_{2^k}}(x)$. Clearly, $x= g^2$ and hence $g \in \Sub_{\mathsf{S}_{2^k}}(x)$.  Now, let us define a block structure $B$ on $\{1,2,\ldots, 2^k\}$ as follows. For each $j \in \{0,\ldots, k\}$, $B$ contains $2^j$ blocks $b_{j1},\ldots, b_{j2^j}$ of size $2^{k-j}$ and two elements $a,c\in \{1,2,\ldots, 2^k\}$ lie in the same block if and only if $a\equiv c \pmod 2^{k-j}$. By definition, we have that $x(B)=B$. Moreover, $\{1,2^{k-1}\}\in B$ and hence $h(B)=B$. Thus, by Theorem \ref{SylowBlock}, there exists $P\in \Syl_2(\mathsf{S}_{2^k})$ such that $x,h\in P$. It follows that  $h \in \Sub_{\mathsf{S}_{2^k}}(x)$ and the claim holds. By the claim, we have that $\mathsf{S}_{2^k}=\langle g,h\rangle\leq  \Sub_{\mathsf{S}_{2^k}}(x)$.
\end{itemize}
\end{itemize}

\underline{Case $o(x)<2^{k-1}$:}  In this case,  we may write $x=yz$, with $\Sup(y)\cap \Sup(z)=\varnothing$ and $\Sub_{\mathsf{S}_{2^{k-1}}}(y)=\mathsf{S}_{2^{k-1}}$.  Reasoning as before, we have that  $\mathsf{S}_{2^{k-1}}\wr \mathsf{S}_2\leq \Sub_{\mathsf{S}_{2^k}}(x)$. Moreover, rearranging the cycles in $x$, we may write $x=gh$, where again $\Sup(g)\cap \Sup(h)=\varnothing$ and  $0<|\Sup(y)\cap \Sup(g)|<2^{k-1}$. We deduce that $\Sub_{\mathsf{S}_{2^{k-1}}}(g)\leq \Sub_{\mathsf{S}_{2^k}}(x)$, but since $\Sub_{\mathsf{S}_{2^{k-1}}}(g)\not\leq \mathsf{S}_{2^{k-1}}\wr \mathsf{S}_2$ we have that  $\Sub_{\mathsf{S}_{2^k}}(x)=\mathsf{S}_{2^k}$.
\end{proof}
\end{thm}

Now, we prove a result which implies Theorem \ref{thmB} for $\mathsf{S}_{2^k}$. The condition (III) of the following result may look a technical condition, but it will be relevant for the proof.

\begin{thm}\label{TheoremB}
Let $k\geq 3$ and let $g \in \mathsf{S}_{2^k}$ be a $2^k$-cycle. For any $2$-element $x \in \mathsf{S}_{2^k}$ there exists  a map 
$$\Gamma:\Irr^x(\mathsf{S}_{2^k})\rightarrow \Irr^x(\Sub_{\mathsf{S}_{2^k}}(x)) $$
satisfying the following properties.
\begin{itemize}
\item [(I)]$\Gamma(\chi)(1)_2=\chi(1)_2$ for every $\chi\in \Irr^x(\mathsf{S}_{2^k})$.

\item[(II)]  $\Gamma(\chi)(x)=\pm \chi(x)$ for every $\chi\in \Irr^x(\mathsf{S}_{2^k})$.

\item[(III)] For each $\chi \in \Irr_{2'}(\mathsf{S}_{2^k})$ there exists $\varepsilon \in \{1,-1\}$ (depending on $\chi$) such that $\Gamma(\chi)(x)=\varepsilon\chi(x) $ and $\Gamma(\chi)(g)=\varepsilon\chi(g) $.
\end{itemize}
\begin{proof}
 We proceed by induction on $k$. For $k=3$, the result can be checked by GAP \cite{gap}. By inductive hypothesis, we assume that the result holds for $k-1$.

If $\Sub_{\mathsf{S}_{2^k}}(x)=\mathsf{S}_{2^k}$, then the result follows trivially. We may assume that $\Sub_{\mathsf{S}_{2^k}}(x)<\mathsf{S}_{2^k}$ and hence, by Theorem \ref{Subnormalizers},  $x$ is either a $2^k$-cycle or $o(x)=2^{k-1}$ and $\fix(x)>0$.

If $x$ is a $2^k$-cycle, then the result follows from Lemma \ref{TypeI}. Thus, we may assume that $x$ can be written as $zy$ for $z$ a $2^{k-1}$-cycle, $y\in \mathsf{S}_{2^{k-1}}$, $\fix(y)=\fix(x)>0$ and  $\Sub_{\mathsf{S}_{2^k}}(x)=\Sub_{\mathsf{S}_{2^{k-1}}}(y)\wr \mathsf{S}_2$. For the remainder of the proof we write $H$ to denote $\Sub_{\mathsf{S}_{2^{k-1}}}(y)$ and $K$ to denote $\Sub_{\mathsf{S}_{2^{k}}}(x)$. Then $K= H\wr \mathsf{S}_2$.

By inductive hypothesis, we have that there exists a bijection
$$\Omega: \Irr^y(\mathsf{S}_{2^{k-1}})\rightarrow \Irr^y(H)$$
satisfying  (I), (II) and (III).

We observe that, since $\Omega$ satisfies  (III), $\Omega(\chi)(z)=\pm\chi(z)=\pm1$ for any $\chi \in \Irr_{2'}(\mathsf{S}_{2^{k-1}})$. Since $\Omega$ satisfies  condition (I), we have that $\Irr_{2'}(H)=\{\Omega(\chi)\mid \chi \in \Irr_{2'}(\mathsf{S}_{2^{k-1}})\}$.

We claim now that $\Irr^z(H)=\Irr_{2'}(H)$. We observe that $\mathbf{C}_{\mathsf{S}_{2^{k-1}}}(z)=\langle z \rangle\subseteq H$. Applying the second orthogonality relation, we have 
\begin{align*}2^{k-1} &=   |\mathbf{C}_{H}(z)|=\sum_{\phi\in \Irr^z(H)}|\phi(z)|^2\geq \\
&\geq \sum_{\phi\in \Irr_{2'}(H)}|\phi(x)|^2=\sum_{\chi\in \Irr_{2'}(\mathsf{S}_{2^{k-1}})}|\Omega(\chi)(z)|^2=2^{k-1}.\end{align*}
This forces $\Irr^z(H)=\Irr_{2'}(H)$ as we claimed.

We make a brief discussion on the characters in $\Irr(K)$. Let  $\chi \in \Irr(K)$ and let $\phi,\psi\in \Irr(H)$ such that $[\phi\times \psi,\chi_{H\times H}]\not=0$.  If $\phi\not=\psi$, then $(\phi\times \psi)^{K}\in \Irr(K)$ and hence $\chi=(\phi\times \psi)^{K}$. Moreover, in this case
$$\chi(x)=\phi(z)\psi(y)+\phi(y)\psi(z).$$ 
 On the other hand, if  $\phi=\psi$, then $\phi\times \phi$ is extendible to  $\Psi\in \Irr(K)$. Thus, $\chi$ is either $\Psi$ or $\Psi\cdot\rho$, where $\Irr(\mathsf{S}_2)=\{1,\rho\}$.

Now, we define  $\Gamma(\chi)$ for $\chi \in \Irr^{x}(\mathsf{S}_{2^k})$. We have to distinguish cases.

\underline{Case $\chi \in \Irr_{2'}(\mathsf{S}_{2^k})$:}  For $0\leq m \leq 2^{k-1}-1$, we write  $\tau(m,k-1)=(2^{k-1}-m,1^m)$. We also write  $\chi^m$ to denote the irreducible character of $\mathsf{S}_{2^{k-1}}$ associated to $\tau(m,k-1)$. We know that $$\Irr_{2'}(\mathsf{S}_{2^{k-1}})=\{\chi^m\mid 0\leq m \leq 2^{k-1}-1\}.$$ We also define $\chi^{m,1}$ and $\chi^{m,-1}$ as the irreducible characters  of $\mathsf{S}_{2^k}$ associated to the partitions (hooks) $\tau(m,k)=(2^{k}-m,1^{m})$ and $\tau(2^{k-1}+m,k)=(2^{k-1}-m,1^{2^{k-1}+m})$, respectively. We observe that $$\Irr_{2'}(\mathsf{S}_{2^{k}})=\{\chi^{m,i}\mid0\leq m \leq 2^{k-1}-1, i \in \{1,-1\}\}.$$

Let $0\leq m\leq 2^{k-2}-1$ and let $a=\chi^{m}(y)$.  By the Murnaghan--Nakayama  rule, we deduce  that 
\begin{align*}& \{(\chi^{t,i}(g),\chi^{t,i}(x))\mid t\in\{m,2^{k-1}-m-1\},i\in \{1,-1\}\}=\\&
=\{(1,a),(1,-a)(-1,a),(-1,-a)\}.\end{align*}
We remark that we had to study the cases $y \in A_{2^{k-1}}$, $y \not\in A_{2^{k-1}}$ and the case $m$ even or $m$ odd,  separately.

Now, let $\Irr(\mathsf{S}_2)=\{1,\rho\}$. Given $t\in\{m,2^{k-1}-m-1\}$, we  have that $\Omega(\chi^t)\times \Omega(\chi^t)$ has an extension $\Psi^t\in \Irr(K)$. We write $\Phi^{t,1}$ to denote $\Psi^t$ and $\Phi^{t,-1}$ to denote $\Psi^t\cdot \rho$. Let $t\in\{m,2^{k-1}-m-1\}$ and let $i \in \{1,-1\}$. We know that $\Phi^{t,i}(y)=(-1)^i$. Moreover, since $\Omega$ satisfies condition (III), we have that $\Omega(\chi^t)(z)=$ and $\Omega(\chi^t)(y)$ for some $\varepsilon \in\{\pm1\}$. It follows that  $\Phi^{t,i}(x)=\Omega(\chi^t)(z) \Omega(\chi^t)(y)=(\varepsilon)^2\chi^t(z) \chi^t(y)=(-1)^t\chi^t(y)$.

It follows that 
\begin{align*}& \{(\Phi^{t,i}(g),\Phi^{t,i}(x))\mid t\in\{m,2^{k-1}-m-1\},i\in \{1,-1\}\}=\\&
=\{(1,a),(1,-a)(-1,a),(-1,-a)\}\end{align*}
for $y \in A_{2^{k-1}}$ and
\begin{align*}& \{(\Phi^{t,i}(g),\Phi^{t,i}(x))\mid t\in\{m,2^{k-1}-m-1\},i\in \{1,-1\}\}=\\&=\{(1,(-1)^ma),(1,(-1)^ma)(-1,(-1)^ma),(-1,(-1)^ma)\}\end{align*}
for $y \not \in A_{2^{k-1}}$.

From this discussion we deduce that there exists a bijection (depending on $m$)
$$f: \{m,2^{k-1}-m-1\}\times \{1,-1\}\to \{m,2^{k-1}-m-1\}\times \{1,-1\}$$
such that either 
$$(\chi^{t,i}(g),\chi^{t,i}(x))=(\Phi^{f(t,i)}(g),\Phi^{f(t,i)}(x))$$ or $$(\chi^{t,i}(g),\chi^{t,i}(x))=(-\Phi^{f(t,i)}(g),-\Phi^{f(t,i)}(x)).$$
Given $m$ with $0\leq m\leq 2^{k-2}-1$ and $t\in\{m,2^{k-1}-m-1\},i\in \{1,-1\}$  we define $\Gamma(\chi^{t,i})=\Phi^{f(t,i)}$. Then $\Gamma$ defines a bijection from $\Irr_{2'}(\mathsf{S}_{2^k})$ to $\Irr_{2'}(K)$, satisfying  (III).

\underline{Case $\chi=\chi^{\lambda}$ for  $\chi^{\lambda}\not \in  \Irr_{2'}(\mathsf{S}_{2^k})$ and $\lambda$ possessing a unique $2^{k-1}$-hook:}  Let $\tau$ be a hook of length $2^{k-1}$ and let $\mu$ be a partition of $2^{k-1}$ which is not a hook. By Lemma \ref{ext}, we know that there exists a unique partition $\lambda(\tau,\mu)$ such that $\lambda(\tau,\mu)$ has $\tau$ as a hook and $\lambda(\tau,\mu)\setminus \tau=\mu$.  In this case, we also have that $\chi^{\lambda(\tau,\mu)}(1)_2=2\chi^{\mu}(1)_2$. Moreover, applying the Murnaghan--Nakayama rule, we have that 
\begin{equation}
\label{eq:1}
\chi^{\lambda(\tau,\mu)}(x)=(-1)^{ht(\tau)}\chi^{\mu}(y)=\chi^{\tau}(z)\chi^{\mu}(y).
\end{equation}
 Thus, $\chi^{\lambda(\tau,\mu)}\in \Irr^x (K)$ if and only if $\chi^{\mu}\in \Irr^y (H)$. Now,  we define $\Gamma(\chi^{\lambda(\tau,\mu)})=(\Omega(\chi^{\tau}),\Omega(\chi^{\mu}))^{K}\in \Irr(K)$. We have
$$\Gamma(\chi^{\lambda(\tau,\mu)})(1)_2=2\Omega(\chi^{\tau})(1)_2\Omega(\chi^{\mu})(1)_2=2\chi^{\mu}(1)_2=\chi^{\lambda(\tau,\mu)}(1)_2,$$
where the second equality holds because $\Omega$ satisfies  (I).

In addition, $$\Gamma(\chi^{\lambda(\tau,\mu)})(x)=\Omega(\chi^{\tau})(z)\Omega(\chi^{\mu})(y)=\pm\chi^{\mu}(y)=\pm\chi^{\lambda(\tau,\mu)}(x),$$ where the second equality holds because $\Omega$ satisfies (II) and (III). Therefore, $\Gamma$ satisfies (I) and (II) for the characters of the form $\chi^{\lambda(\tau,\mu)}$, where $\tau$ is a hook and $\mu$ is not a hook.

\underline{Case $\chi=\chi^{\lambda}$ for $\lambda$ having two $2^{k-1}$-hooks:} By Lemma \ref{Doublehook},  if $\lambda$ is such a partition, then $\lambda$ has the form $\gamma(a,b,k)=(2^{k-1}-a,2^{k-1}-b+1,2^a,1^{b-a-1})$ for $0\leq a<b\leq 2^{k-1}-1$. We know that $\tau(a+1,k-1)$ and $\tau(b-1,k-1)$ are $2^{k-1}$-hooks of $\gamma(a,b,k)$. Moreover, 
$$ \gamma(a,b,k)\setminus \tau(a+1,k-1)=\tau(b,k-1) \text{ and } \gamma(a,b,k)\setminus \tau(b-1,k-1)=\tau(a,k-1).$$
In addition, $\chi^{\gamma(a,b,k)}(1)_2=2$ by Lemma \ref{Doublehook}.

Given $0\leq a<b\leq 2^{k-1}-1$, we write $\chi^{(a,b)}$ to denote the character $\chi^{\gamma(a,b,k)}\in \Irr(\mathsf{S}_{2^k})$.  By the Murnaghan--Nakayama rule, we have that

\begin{equation}
\label{eq:2}
 \begin{aligned}
\chi^{(a,b)}(x)&=\chi^{a+1}(z)\chi^{b}(y)+\chi^{b-1}(z)\chi^{a}(y)=\\& =(-1)(\chi^{a}(z)\chi^{b}(y)+\chi^{b}(z)\chi^{a}(y)).
 \end{aligned}
\end{equation}

We define $\Gamma(\chi^{(a,b)})=(\Omega(\chi^{a}),\Omega(\chi^{b}))^{K}\in \Irr(K)$. Since $\Omega$ satisfies  (I), we deduce that $\Gamma(\chi^{(a,b)})(1)_2=2\Omega(\chi^{a})(1)_2\Omega(\chi^{b})(1)_2=2$. Moreover, since  $\Omega$ satisfies  (III), we know that, for each $i \in \{a,b\}$ there exists $\varepsilon_i\in \{1,-1\}$ such that $\Omega(\chi^{i})(z)=\varepsilon_i\chi^{i}(z)$ and $\Omega(\chi^{i})(y)=\varepsilon_i\chi^{i}(y)$. It follows that 
\begin{align*}\Gamma(\chi^{(a,b)})(x)& =\Omega(\chi^{a})(z)\Omega(\chi^{b})(y)  +\Omega(\chi^{a})(z)\Omega(\chi^{b})(y)=\\ &=\varepsilon_a\varepsilon_b(\chi^{a}(z)\chi^{b}(y)+\chi^{b}(z)\chi^{a}(y))=\pm \chi^{(a,b)}(x).\end{align*}

Thus, $\Gamma$ satisfies (I) and (II) for these elements.

Therefore, we have defined a bijection  from $\Irr^x(\mathsf{S}_{2^k})$ to $\Irr^x(K)$  satisfying  (I), (II) and (III).
\end{proof}
\end{thm}

Now, we work towards a general proof of Theorem \ref{thmB}. First, we have to introduce results to determine the structure of subnormalizers of $2$-elements in $\mathsf{S}_n$ for $n\neq 2^k$.  Given $\Omega \subseteq \{1,\ldots,n\}$, we write $\overline{\Omega}:=\{1,\ldots,n\}\setminus \Omega$. In particular $\mathsf{S}_{\Omega}\times \mathsf{S}_{\overline{\Omega}}$ is a maximal subgroup of $\mathsf{S}_n$.

\begin{lem}\label{Reduction1}
Let $\ell\geq 1$, let $m \geq 1$ and let $y \in \mathsf{S}_{\ell2^m}$ be a product of $\ell$ cycles of length $2^m$. The following holds
\begin{itemize}

\item[(i)] If $\ell=1$, then $\Sub_{\mathsf{S}_{2^m}}(y)=P_{2^m}$, where $P_{2^m}\in \Syl_2(\mathsf{S}_{2^m})$ with $y \in P_{2^m}$. 

\item[(ii)] If $\ell>1$, then $\Sub_{\mathsf{S}_{\ell 2^m}}(y)= \mathsf{S}_{\ell 2^m}$.
\end{itemize}
\begin{proof}
We prove the result by induction on $\ell$. 

 If $\ell=1$, then $\ell2^m=2^m$ and $y$ is a $2^m$-cycle. Then $y$ is a picky element of $\mathsf{S}_{2^m}$ by Proposition \ref{ClassificationPickySymmetric}. Thus,  $\Sub_{\mathsf{S}_{2^m}}(x)=P_{2^m}$ and the result follows in this case.

Now, let us assume that $\ell>1$ and the result holds for every $t<\ell$.  Let $\ell 2^m=2^{n_1}+\ldots+2^{n_f}$ with $n_f>n_{f-1}>\ldots>n_1\geq m$. Assume first that $f=1$. Since $\ell>1$, we deduce that $y$ is not a $2^{n_f}$-cycle and hence $\Sub_{\mathsf{S}_{\ell 2^m}}(y)=\mathsf{S}_{\ell 2^m}$ by Theorem \ref{Subnormalizers}.

Now, let us assume that $f>1$. In particular $\ell\geq 3$.    Let $r=2^{n_1-m}$. Then we write $y=xz$, where $x$ is a product of $\ell-r$ cycles of length $m$, $z$ is a product of $r$ cycles of length $2^m$ and $\sup(x)\cap\Sup(z)=\varnothing$. Let $\Omega=\Sup(\Omega)$ and let $P_{\overline{\Omega}}\in \mathsf{S}_{\overline{\Omega}}$ with $z\in P_{\overline{\Omega}}$. 
By Lemma \ref{SylowBlock}, we deduce that 
$$\Sub_{\mathsf{S}_{\Omega}}(x)\times \Sub_{\mathsf{S}_{\overline{\Omega}}}(z)\leq \Sub_{\mathsf{S}_{\ell2^m}}(y).$$
We also have that $1<\ell-r<\ell$ and hence $\Sub_{\mathsf{S}_{\Omega}}(x)=\mathsf{S}_{\Omega}$. Moreover, $P_{\overline{\Omega}}\leq \Sub_{\mathsf{S}_{\overline{\Omega}}}(z)$. Thus, 
$$\mathsf{S}_{\Omega}\times P_{\overline{\Omega}}\leq  \Sub_{\mathsf{S}_{\ell2^m}}(y).$$
Seeking for a contradiction, let us assume that $\Sub_{\mathsf{S}_{\ell2^m}}(y)<\mathsf{S}_{\ell2^m}$. We notice that, since $|\overline{\Omega}|=r\cdot 2^m=2^{n_1}$, then $P_{\overline{\Omega}}$ acts transitively on $\overline{\Omega}$. Therefore, the unique maximal subgroup of $\mathsf{S}_{\ell2^m}$ containing $\Sub_{\mathsf{S}_{\ell2^m}}(y)$ is $\mathsf{S}_{\Omega}\times \mathsf{S}_{\overline{\Omega}}$. Now, let us write $y=x_1z_1$ where  $x_1$ is a product of $\ell-r$ cycles of length $m$, $z_1$ is a product of $r$ cycles of length $2^m$, $\sup(x_1)\cap\Sup(z_1)=\varnothing$ and $\Sup(x_1)\neq \Sup(x)$. Let $\Delta=\Sup(x_1)$. Reasoning as above, we deduce that $\mathsf{S}_{\Delta}\leq \Sub_{\mathsf{S}_{\ell2^m}}(y)$ but $\mathsf{S}_{\Delta}\not \leq \mathsf{S}_{\Omega}\times \mathsf{S}_{\overline{\Omega}}$, which is a contradiction. This contradiction shows that  $\Sub_{\mathsf{S}_{\ell2^m}}(y)=\mathsf{S}_{\ell2^m}$.
\end{proof}
\end{lem}

\begin{thm}\label{Reduction2}
Let $n$ be an integer such that $n$ is not a power of $2$ and let $x\in \mathsf{S}_{n}$ be a $2$-element. Assume that the lengths of the cycles in $x$ are $2^{m_i}$ for some integers $0\leq m_1\leq m_2\ldots \leq m_r$ and that $x$ possesses $\ell_i\geq 1$ cycles of length $2^{m_i}$ for every $1\leq i \leq r$. Then $\Sub_{\mathsf{S}_n}(x)<\mathsf{S}_n$ if and only if  there exists $t\geq 2$ such that $2^{m_t}>\sum_{i=1}^{t-1}\ell_i2^{m_i}$. Moreover, in such a case
$$\Sub_{\mathsf{S}_n}(x)=\Sub_{\mathsf{S}_{k}}(y)\times \Sub_{\mathsf{S}_{n-k}}(z),$$
where $k=2^{m_r}+\ldots+2^{m_t}$, $y$ is the product of all cycles of $x$ of length at least $2^{m_t}$ and $z$ is the product of all cycles of $x$ of length smaller than $2^{m_t}$.
\begin{proof}
First, let us assume that there exists $t\geq 2$ such that $2^{m_t}>\sum_{i=1}^{t-1}\ell_i2^{m_i}$ and let us write $m=m_t$. Let $n= \sum_{i=1}^f2^{n_i}$ for $0\leq n_1<n_2< \ldots < n_f$ and let us assume that $n_{s-1}<m\leq n_s$. Our hypotheses imply that $k=2^{m_r}+\ldots+2^{m_t}=2^{n_f}+\ldots+2^{n_s}$ and $n-k=2^{n_{s-1}}+\ldots+2^{n_1}$.

Without loss of generality, we may assume that $\Sup(y)=\{1,2,\ldots ,k\}$.
Let $B_1$ be a block structure of $\{1,\ldots,k\}$ such that $y(B_1)=B_1$ and let $B_2$ be a block structure of $\{k+1,\ldots,n\}$ such that $z(B_2)=B_2$. Then $B=B_1\cup B_2$ is a block structure of $\{1,\ldots,n\}$ such that $x(B)=B$.

We claim that if  $B$ is  a block structure of $\{1,\ldots,n\}$ such that $x(B)=B$, then $B=B_1\cup B_2$, where  $B_1$ is a block $\{1,\ldots,k\}$ such that $y(B_1)=B_1$ and  $B_2$ is a block structure of $\{k+1,\ldots,n\}$ such that $z(B_2)=B_2$. Seeking for a contradiction, let us assume that there exists $a \leq k$ such that $a$ does not lie in any $b\in B$ with $|b|=2^m$. It is easy to see that $x(a)$ does not lie in any $b\in B$ with $|b|=2^m$ (otherwise $a \in x^{-1}(b)\in B$). Reasoning similarly, we have that the elements $a,x(a),\ldots, x^{2^m}(a)$ are pairwise different and none of them lie in any subset  $b \in B$ with  $|b|=2^m$. This is impossible since $2^m>n-k$.

From the two paragraphs above we deduce that $\Sub_{\mathsf{S}_n}(x)=\Sub_{\mathsf{S}_{k}}(y)\times \Sub_{\mathsf{S}_{n-k}}(z)$.

Now, we prove that $\Sub_{\mathsf{S}_n}(x)=\mathsf{S}_n$ if $x$ satisfies that $2^{m_t}\leq \sum_{i=1}^{t-1}\ell_i2^{m_i}$ for every $2\leq t \leq r$.  We proceed by induction on $r$. For $r=1$ the result follows from Lemma \ref{Reduction1}. Let us assume that the result holds for $r-1$.

Now, let us assume that $2^{m_t}\leq \sum_{i=1}^{t-1}\ell_i2^{m_i}$ for every $2\leq t \leq r$.   Let us write $x= y\cdot z$, where $y$ is the product of the $\ell_r$ cycles of length $2^{m_r}$ and $z$ is the product of the rest of the cycles in $x$. It is easy to see that 
$$\Sub_{\mathsf{S}_{\ell_r2^{m_r}}}(y)\times \Sub_{\mathsf{S}_{n-\ell_r2^{m_r}}}(z)\leq  \Sub_{\mathsf{S}_n}(x).$$
We analyse the case $\ell_r=1$ and the case $\ell_r>1$ separately:

\underline{Case  $\ell_r=1$:} Let $\Omega=\Sup(y)$ and let $Q\in \Syl_2(\mathsf{S}_{\Omega})$ with $y \in Q$. 	It is easy to see that $\Sub_{\mathsf{S}_{\Omega}}(y)=Q$. We claim that $\Sub_{\mathsf{S}_{\overline{\Omega}}}(z)=\mathsf{S}_{\overline{\Omega}}$. If $n-k$ is not a power of $2$, then  $z$ satisfies that $2^{m_t}\leq \sum_{i=1}^{t-1}\ell_i2^{m_i}$ for every $2\leq t \leq r-1$ and hence $\Sub_{\mathsf{S}_{\overline{\Omega}}}(z)=\mathsf{S}_{\overline{\Omega}}$ by inductive hypothesis.  Now, let us assume that $n-k$ is a power of $2$. Since  $2^{m_r}\leq \sum_{i=1}^{r-1}\ell_i2^{m_i}$ we deduce that $m_r<n_f$ and hence $n-k=2^{n_f}$.  Now, since $o(z)<2^{m_r}<2^{n_f}$ then $\Sub_{\mathsf{S}_{\overline{\Omega}}}(z)=\mathsf{S}_{\overline{\Omega}}$ by Theorem \ref{Subnormalizers}.

From the claim we deduce that
$$Q\times \mathsf{S}_{\overline{\Omega}}\leq  \Sub_{\mathsf{S}_n}(x).$$
Thus, if $M$ is a maximal subgroup of $\mathsf{S}_n$ containing $\Sub_{\mathsf{S}_n}(x)$, then $M=\mathsf{S}_{\Omega}\times \mathsf{S}_{\overline{\Omega}}$.

Now, the hypothesis $2^{m_r}\leq \sum_{i=1}^{r-1}\ell_i2^{m_i}$ implies that $n_f>m_r$ and hence there exists a block structure $B$ of $\{1,\ldots,n\}$ such that, $x(B)=B$ and there exists $b \in B$ with $|b|=2^{n_f}$ and $\Omega\subset b$. Thus, $\Sub_{\mathsf{S}_n}(x)$ contains a cycle $y_1$ with $\Omega\subset \Sup(y_1)=b$. It follows that $\Sub_{\mathsf{S}_n}(x)\not \leq \mathsf{S}_{\Omega}\times \mathsf{S}_{\overline{\Omega}}$ and hence $\Sub_{\mathsf{S}_n}(x)=\mathsf{S}_n$.

\underline{Case  $\ell_r>1$:} In this case, $\Sub_{\mathsf{S}_{\Omega}}(y)=\mathsf{S}_{\Omega}$. We study the case when $n-k$ is a power of $2$ and the case when $n-k$ is not a power of $2$ separately.

\begin{itemize}

\item \underline{Case $n-k$ is not a power of $2$:} Since $z$ satisfies that $2^{m_t}\leq \sum_{i=1}^{t-1}\ell_i2^{m_i}$ for every $2\leq t \leq r-1$ we have $\Sub_{\mathsf{S}_{\overline{\Omega}}}(z)=\mathsf{S}_{\overline{\Omega}}$ by inductive hypothesis. Thus, $\mathsf{S}_{\Omega}\times \mathsf{S}_{\overline{\Omega}}\leq \Sub_{\mathsf{S}_n}(x)$. Now, let us write $x=y_1z_1$, where $y_1$ possesses exactly $\ell_{r}-1$ cycles of length $2^{m_r}$, $\Sup(y_1)\subseteq \Delta$, $\Sup(z_1)\subseteq \overline{\Delta}$ for some $\Delta\subset \{1,\ldots,n\}$ with $|\Delta|=\ell_r2^{m_r}$. Let $P \in \Syl_2(\mathsf{S}_{\Delta})$ such that $y_1\in P$. Then $P\leq   \Sub_{\mathsf{S}_n}(x)$ but $P \not \leq \mathsf{S}_{\Omega}\times \mathsf{S}_{\overline{\Omega}}$. This forces $ \Sub_{\mathsf{S}_n}(x)= \mathsf{S}_n$.

\item \underline{Case $n-k$ is a power of $2$:} In this case, $n-k=2^h$ for $h\geq m_r$. We also know that $o(z)\leq 2^{m_r-1}\leq 2^{h-1}$. If $o(z)\leq 2^{h-2}$, then $\Sub_{\mathsf{S}_{\overline{\Omega}}}(z)=\mathsf{S}_{\overline{\Omega}}$ and the result follows reasoning as in the previous case. Thus, we may assume that $o(z)=2^{h-1}$, which forces $m_{r}=h$. Let $Q \in \Syl_2(\mathsf{S}_{\overline{\Omega}})$ such that $z\in Q$. Then $\mathsf{S}_{\Omega}\times Q\leq  \Sub_{\mathsf{S}_n}(x)$.

Now, let us write $x=y_1z_1$, where $y_1$ possesses exactly $\ell_{r}-1$ cycles of length $2^{m_r}$, $\Sup(y_1)\subseteq \Delta$, $\Sup(z_1)\subseteq \overline{\Delta}$ for some $\Delta\subset \{1,\ldots,n\}$ with $|\Delta|=\ell_r2^{m_r}$. We have that either $\Sub_{\mathsf{S}_{\Delta}}(y_1)=\mathsf{S}_{\Delta}$ or $\ell_r=2$. If $\Sub_{\mathsf{S}_{\Delta}}(y_1)=\mathsf{S}_{\Delta}$, then $\mathsf{S}_{\Delta}, \mathsf{S}_{\Omega}\times Q\leq  \Sub_{\mathsf{S}_n}(x)$ and hence $ \Sub_{\mathsf{S}_n}(x)= \mathsf{S}_n$ (we notice that $\mathsf{S}_{\Delta}\not \leq \mathsf{S}_{\Omega}\times \mathsf{S}_{\overline{\Omega}}$). Thus, we may assume that $\ell_r=2$. and hence $n=2^{m_r+1}+2^{m_r}$.  Now, let us write  $x=y_2z_2$, where $y_2$ possesses only  one  cycle of length $2^{m_r}$, $\Sup(y_2)\subseteq \Lambda$, $\Sup(z_2)\subseteq \overline{\Lambda}$ for some $\Lambda\subset \{1,\ldots,n\}$ with $|\Lambda|=2\cdot2^{m_r}$ and $\Lambda\neq \Delta$. Let $P_1, Q_1,P_2$ and $Q_2$ be Sylow $2$-subgroups of $\mathsf{S}_{\Delta}, \mathsf{S}_{\overline{\Delta}}, \mathsf{S}_{\Lambda}$ and $\mathsf{S}_{\overline{\Lambda}}$ such that $y_1\in P_1$, $z_1\in Q_1$, $y_2\in P_2$ and $z_2\in Q_2$. It follows that $\mathsf{S}_{\Omega}\times Q,P_1\times Q_1,P_2\times Q_2\leq \Sub_{\mathsf{S}_n}(x)$, which forces $ \Sub_{\mathsf{S}_n}(x)= \mathsf{S}_n$.
\end{itemize}
Thus, the result holds.
\end{proof}
\end{thm}

Now, we are ready to prove Theorem \ref{thmB}.

\begin{proof}[Proof of Theorem \ref{thmB}]
Let $n=2^{n_1}+\ldots+2^{n_f}$ with $n_f>n_{f-1}>\ldots>n_1\geq 0$. We prove the result by induction on $f$. If $f=1$, then $n=2^k$ and the result follows from Theorem \ref{TheoremB}.

 Let us assume that $f\geq 2$ and that the result holds for every $k<f$.   Let  $x \in \mathsf{S}_n$ be a  $2$-element. If $\Sub_{\mathsf{S}_n}(x)=\mathsf{S}_n$, the result follows trivially. Thus, we may assume that $\Sub_{\mathsf{S}_n}(x)<\mathsf{S}_n$. From Theorem \ref{Reduction2} we deduce that there exists $t\geq 2$ such that $2^{m_t}>\sum_{i=1}^{t-1}\ell_i2^{m_i}$. Let us write $k=2^{m_r}+\ldots+2^{m_t}$ and  $h=n-k$.  Let us assume that $n_{s-1}<m_t\leq n_s$. Our hypotheses imply that $k=2^{m_r}+\ldots+2^{m_t}=2^{n_f}+\ldots+2^{n_s}$ and $h=n-k=2^{n_{s-1}}+\ldots+2^{n_1}$. We notice that the number of summands appearing in the $2$-adic expansion of both $k$ and $h$ is smaller than $f$.

Let  $y\in \mathsf{S}_k$ be the product of all cycles of length at least $2^{m_t}$ and let  $z\in \mathsf{S}_h$ be the product of all cycles of length smaller than $2^{m_t}$. By inductive hypothesis, we have that there exist a bijection
$$\Phi: \Irr^{y}(\mathsf{S}_{k}) \rightarrow \Irr^{y}(\Sub_{\mathsf{S}_{k}}(y))$$
satisfying the following properties:

\begin{itemize}
\item [(I)] $\Phi(\psi)(1)_2=\psi(1)_2$ for every $\psi \in \Irr^{y}(\mathsf{S}_{k}) $.

\item[(II)] $\Phi(\psi)(y)=\pm \psi(y)$  for every $\psi \in \Irr^{y}(\mathsf{S}_{k}) $.

\end{itemize}

Moreover, there exists a bijection
$$\Delta: \Irr^z(\mathsf{S}_{h})\rightarrow \Irr^{z}(\Sub_{\mathsf{S}_{h}}(z))$$
satisfying the following properties: 
\begin{itemize}
\item [(I)] $\Delta(\phi)(1)_2=\phi(1)_2$ for every $\phi \in \Irr^{z}(\mathsf{S}_{h}) $.

\item[(II)] $\Delta(\phi)(z)=\pm \phi(z)$  for every $\phi \in \Irr^{z}(\mathsf{S}_{h}) $.

\end{itemize}

Now, let $\lambda$ be a partition of $n$ and let $\mu=C_{2^{m_t}}(\lambda)$. We notice that  if $\chi^\lambda(x)\neq0$  then $|\mu|=h$ and $\chi^{\mu}(z)\neq0 $. Let $T_C(\lambda)=\{\lambda_{ij}\mid0\leq i\leq n_f, 0\leq j \leq 2^i-1\}$ be the $2$-core tower of $\lambda$. Given $0\leq i \leq n_f$ and $ 0\leq j \leq 2^i-1$ we define $\mu_{ij}=\lambda_{ij}$ for $i<s$ and $0$ otherwise. We notice that $T_C(\mu)=\{\mu_{ij}\mid0\leq i\leq n_f, 0\leq j \leq 2^i-1\}$.  Analogously, we define $\gamma_{ij}=\lambda_{ij}$ for $j\geq s$ and $0$ otherwise. We consider $\gamma$ as the unique partition of $\ell 2^m$ such that $T_C(\gamma)=\{\gamma_{ij}\mid0\leq i\leq n_f, 0\leq j \leq 2^i-1\}$. We notice that
$$\chi^{\lambda}(x)=\chi^{\gamma}(y)\chi^{\mu}(z)$$
and 
$$\chi^{\lambda}(1)_2=\chi^{\gamma}(1)_2\chi^{\mu}(1)_2.$$

By Theorem \ref{Reduction2} we have that $\Sub_{\mathsf{S}_n}(x)=\Sub_{\mathsf{S}_{k}}(y)\times \Sub_{\mathsf{S}_{n-k}}(z)$.  Now, we define
$$\Gamma:\Irr^x(\mathsf{S}_n)\rightarrow \Irr^x(\Sub_{\mathsf{S}_n}(x))$$
 by
$$\Gamma(\chi^{\lambda})=\Phi(\chi^{\gamma})\times \Delta(\chi^{\mu}).$$
Since $\Phi$ and $\Delta$ satisfy conditions (I) and (II), we deduce that 
$$\Gamma(\chi^{\lambda})(1)_2=\Phi(\chi^{\gamma})(1)_2\Delta(\chi^{\mu})(1)_2= \chi^{\gamma}(1)_2\chi^{\mu}(1)_2=\chi^{\lambda}(1)_2$$
and 
$$\Gamma(\chi^{\lambda})(x)=\Phi(\chi^{\gamma})(y)\Delta(\chi^{\mu})(z)=\pm \chi^{\gamma}(y)\chi^{\mu}(z)=\pm\chi^{\lambda}(x).$$
Thus, $\Gamma$ is a bijection satisfying the desired properties.
\end{proof}

\section{Picky conjecture}\label{PickyPart}

The goal of this section is to prove Theorem \ref{thmA}.  We begin by studying the case of picky $2$-elements of Type II. From Theorem \ref{thmB}, we deduce that the Picky Conjecture holds for these elements. However, we will calculate the exact values of $\chi(x)$, where $x\in \mathsf{S}_{2^k}$ is a picky $2$-element of Type II.

\begin{cor}\label{CasePower2}
Let $k\geq 3$ and let $g \in \mathsf{S}_{2^k}$ be a $2^k$-cycle. Let $x \in \mathsf{S}_{2^k}$ be a picky $2$-element of Type II and let $P_{2^k}\in \Syl_2(\mathsf{S}_{2^k})$ be the unique Sylow $2$-subgroup of $\mathsf{S}_{2^k}$ containing $x$. Then there exists  a map 
$$\Gamma:\Irr^x(\mathsf{S}_{2^k})\rightarrow \Irr^x(P_{2^k}) $$
satisfying the following properties:
\begin{itemize}
\item [(I)] $\Gamma(\chi)(1)_2=\chi(1)_2$ for every $\chi\in \Irr^x(\mathsf{S}_{2^k})$.

\item[(II)] $\Gamma(\chi)(x)=\pm \chi(x)$ for every $\chi\in \Irr^x(\mathsf{S}_{2^k})$.

\item[(III)] For each $\chi \in \Irr_{2'}(\mathsf{S}_{2^k})$ there exists $\varepsilon \in \{1,-1\}$ (depending on $\chi$) such that  $\Gamma(\chi)(x)=\varepsilon\chi(x) $ and $\Gamma(\chi)(g)=\varepsilon\chi(g) $.
\end{itemize}
Moreover, for $H\in \{\mathsf{S}_{2^k},P_{2^k}\}$ the following hold:
\begin{itemize}

\item[a)]If  $\chi \in \Irr_{2'}(H)$, then $\chi(x)=\pm1$.

\item[b)] If  $\chi \in \Irr^x(H) \setminus\Irr_{2'}(H)$, then $\chi(x)=\pm2$.
\end{itemize}
\begin{proof}
Since $\Sub_{\mathsf{S}_{2^k}}(x)=\mathbf{N}_{\mathsf{S}_{2^k}}(P_{2^k})=P_{2^k}$, the existence of the bijection $\Gamma$ follows by Theorem \ref{TheoremB}. We notice that, since $\Gamma$ satisfies (II) it is enough to prove that statements a) and b) hold for $\mathsf{S}_{2^k}$. Now, we prove that statements a) and b) hold by induction on $k$.

Assume first that $k=3$. In this case $x=(1,2,3,4)(5,6)$ and $\Irr^x(\mathsf{S}_{8})=\Irr_{2'}(\mathsf{S}_{8})\cup \{\chi^{(4,2,1,1)},\chi^{(3,3,2)}\}$.  The result holds by checking that $\chi^{(4,2,1,1)}(x)=2$,   $\chi^{(3,3,2)}(x)=-2$ and that $\chi(x)=\pm1$ for every $\chi \in \Irr_{2'}(\mathsf{S}_8)$.

Assume that the result holds for $k-1$ and we prove the result for $k$.   Let us write $x=zy$,where $z$ is a $2^{k-1}$-cycle and $y$ is a picky $2$-element of Type II in $\mathsf{S}_{2^{k-1}}$.  We follow the constructions and the notation introduced in Theorem \ref{TheoremB}. 

\underline{Case $\chi \in \Irr_{2'}(\mathsf{S}_{2^k})$:} We recall that given $0 \leq m \leq 2^{k-1}-1$, the characters $\chi^{m,1}$ and $\chi^{m,-1}$ are the characters of $\mathsf{S}_{2^k}$ associated to the hooks $\tau(m,k)=(2^{k}-m,1^{m})$ and $\tau(2^{k-1}+m,k)=(2^{k-1}-m,1^{2^{k-1}+m})$, where $\chi^m$ is the character of $\mathsf{S}_{2^{k-1}}$ associated to $\tau(m,k-1)=(2^{k-1}-m, 1^m)$. Then $$\Irr_{2'}(\mathsf{S}_{2^k})=\{\chi^{m,i}\mid0 \leq m \leq 2^{k-1}-1, i \in \{1,-1\}\}.$$ For  any $0 \leq m \leq 2^{k-1}-1$, we have that $\chi^{m,1}(x)=\chi^m(y)$ and $\chi^{m,-1}(x)=(-1)\chi^m(y)$, by the Murnaghan--Nakayama rule. By the inductive hypothesis, we have that $\chi^m(y) =\pm1$ and hence $\chi^{m,1}(x),\chi^{m,-1}(x)=\pm1$. Thus, a) holds.

\underline{Case $\chi^{\lambda}\not \in \Irr_{2'}(\mathsf{S}_{2^k})$ and  $\lambda$ has a unique $2^{k-1}$-hook:} By Lemma \ref{ext}, $\lambda=\lambda(\tau,\mu)$, where $\tau$  is a hook of length $2^{k-1}$ and  $\mu$ is a partition of $2^{k-1}$ which is not a hook.  By (\ref{eq:1}), we have that $\chi^{\lambda(\tau,\mu)}(x)=(-1)^{ht(\tau)}\chi^{\mu}(y)$. By the inductive hypothesis $\chi^{\mu}(y)=\pm2$ and hence $\chi^{\lambda(\tau,\mu)}(x)=\pm2$.

\underline{Case $\chi^{\lambda} \in \Irr^{x}(\mathsf{S}_{2^k})$ and  $\lambda$ has two $2^{k-1}$-hooks:} In this case $\lambda=\gamma(a,b,k)$, where $\gamma(a,b,k)=(2^{k-1}-a,2^{k-1}-b+1,2^a,1^{b-a-1})$ for $0\leq a<b\leq 2^{k-1}-1$. By (\ref{eq:2}), we have
$$\chi^{\lambda}(x)=\chi^{a+1}(z)\chi^{b}(y)+\chi^{b-1}(z)\chi^{a}(y).$$
Since $z$ is a $2^{k-1}$-cycle, we have that $\chi^{a+1}(z), \chi^{b-1}(z)=\pm1$. Moreover, by inductive hypothesis $\chi^{a}(y),\chi^{a}(y)=\pm 1$. It follows that $\chi^{\lambda}(x)\in\{-2,0,2\}$. Since $\chi^{\lambda}\in \Irr^x(\mathsf{S}_{2^k})$, we deduce that $\chi^{\lambda}(x)=\pm2$.
\end{proof}
\end{cor}

Before continuing, we make a brief discussion which was suggested  to the author  by Alexander Moretó.  Huppert \cite{Huppert} showed that, for $k \geq 3$ and $P_{2^k}\in \Syl_2(\mathsf{S}_{2^k})$ we have
$$\{\psi(1)\mid\psi\in \Irr(P_{2^k})\}=\{2^j \mid 0\leq j \leq 2^{k-2}+2^{k-3}-1\}.$$
In contrast,  the set $\{\chi(1)_2\mid\chi\in \Irr(\mathsf{S}_{2^k})\}$ has ``gaps''. For example,
$$\{\chi(1)_2\mid\chi\in \Irr(\mathsf{S}_{8})\}=\{1,2,4,8,64\}.$$

The following corollary shows that, for $x\in\mathsf{S}_{2^k}$ a picky element of Type II, the set  $\{\chi(1)_2\mid\chi\in \Irr^x(\mathsf{S}_{2^k})\}$ has a better behaviour.

\begin{cor}\label{GoodBehaviour}
Let $k\geq 3$ and let  $x \in \mathsf{S}_{2^k}$ be a picky element of Type II.  We have that  $$\{\chi(1)_2\mid\chi\in \Irr^x(\mathsf{S}_{2^k})\}=\{1,2,\ldots,2^{k-2}\}.$$
In addition, 
$$|\{\chi\in \Irr^x(\mathsf{S}_{2^k})\mid\chi(1)_2=2^{k-2}\}|= 2^{k-2}.$$
\end{cor}

Before proving Corollary \ref{GoodBehaviour}, we need to prove the following easy result.

\begin{lem}\label{linear}
Let $k\geq 3$, let $x,y\in \mathsf{S}_{2^k}$ be picky elements of Types I and II,  respectively, and let   $(a,b)\in \{(1,1),(1,-1),(-1,1),(-1,-1)\}$. Then
$$|\{\chi \in \Irr_{2'}(\mathsf{S}_{2^k})\mid(\chi(x), \chi(y))=(a,b)\}|=2^{k-2}.$$
\begin{proof}
 For $k=2$, the result is trivial. Assume that the result holds for $k-1$.  Let us write $y=y_1y_2$, where $y_1$ is a $2^{k-1}$-cycle and $y_2$ is a picky element of Type II in $\mathsf{S}_{2^{k-1}}$. 
Applying the Murnaghan--Nakayama rule, we have that 
$$(\chi^{m,1}(x),\chi^{m,1}(y))=(\chi^{m}(y_1), \chi^{m}(y_2) )$$
and that
$$(\chi^{m,-1}(x),\chi^{m,-1}(y))=(\chi^{m}(y_1), (-1) \chi^{m,-1}(y_2) ).$$
Thus, we have that the set $\{\chi \in \Irr_{2'}(\mathsf{S}_{2^k})\mid(\chi(x), \chi(y))=(a,b)\}$ coincides with 
$$\{\chi^{m,1}\mid(\chi^{m}(y_1), \chi^{m}(y_2) )=(a,b)\}\cup\{\chi^{m,-1}\mid(\chi^{m}(y_1), \chi^{m,-1}(y_2) )=(a,-b)\}.$$
Now, by inductive hypothesis we have that 
$$|\{\chi^{m,1}\mid(\chi^{m}(y_1), \chi^{m}(y_2) )=(a,b)\}|=|\{\chi^{m,-1}\mid(\chi^{m}(y_1), \chi^{m,-1}(y_2) )=(a,-b)\}|=2^{k-3}$$
and the result follows.
\end{proof}
\end{lem}

\begin{proof}[Proof of Corollary \ref{GoodBehaviour}]
We prove the result by induction on $k$.  If $k=3$, then the results follows by an easy  inspection.

Let $x\in \mathsf{S}_{2^k} $  be a picky element of Type II. Let $x=z\cdot y$, where $z$ is a $2^{k-1}$-cycle and $y\in \mathsf{S}_{2^{k-1}}$ be a picky element of Type II. By inductive hypothesis, we have that 
$$\{\chi(1)_2\mid\chi\in \Irr^y(\mathsf{S}_{2^{k-1}})\}=\{1,2,\ldots,2^{k-3}\}.$$
Let $0\leq a\leq 2^{k-2}$. Our goal is to prove that there exists $\chi \in \Irr(\mathsf{S}_{2^k})$ such that $\chi(x)\neq 0$ and $\chi(1)_2=2^a$. We divide the proof in different subcases.

\underline{Case $a=0$:} Since $1_{\mathsf{S}_{2^k}}\in \Irr^x(\mathsf{S}_{2^k})$, the result holds in this case.

\underline{Case $a=1$:} By Lemma \ref{linear}, there exists $0<b\leq 2^{k-1}-1$ such that $\chi^b(y)\chi^b(z)>0$. Let us consider the partition $\lambda=\gamma(0,b,k)=(2^{k-1},2^{k-1}-b+1,1^{b-1})$. We notice that $\chi^{\lambda}(1)_2=2$ by Lemma \ref{Doublehook}. Moreover, by  (\ref{eq:2}), we have that
$$\chi^{\lambda}(x)=\chi^{1}(z)\chi^{b}(y)+\chi^{b-1}(z)\chi^{0}(y)=(-1)(\chi^{b}(y)+\chi^{b}(z))\neq 0$$
and the result holds in this case.

\underline{Case $2\leq a\leq k-2$:} From Lemma \ref{ext2} and the Murnaghan--Nakayama rule, we deduce that $\chi^{\lambda}\in \Irr^x(\mathsf{S}_{2^k})$ and $\chi^{\lambda}(1)_2=2^a$ if and only if $\lambda=\lambda(\tau,\mu)$, where $\tau$ is a  $2^{k-1}$-hook, $\chi^{\mu}\in \Irr^y(\mathsf{S}_{2^{k-1}})$ and $\chi^{\mu}(1)_2= 2^{a-1}$. By inductive hypothesis, we have that there exists a partition $\mu$ of $2^{k-1}$ such that $\chi^{\mu}\in \Irr^y(\mathsf{S}_{2^{k-1}})$ and $\chi^{\mu}(1)_2= 2^{a-1}$.  Thus, the result holds in this case.

Since there exist $2^{k-1}$ different hooks of length $2^{k-1}$, then using the above argument, we deduce
$$|\{\chi\in \Irr^x(\mathsf{S}_{2^k})\mid \chi(1)_2=2^{k-2}\}|=2^{k-1}|\{\psi\in \Irr^y(\mathsf{S}_{2^{k-1}})\mid \psi(1)_2=2^{k-3}\}|.$$
Thus, the ``in addition'' part follows by an easy inductive argument.
\end{proof}

Now, we study the case when $p=2$ and $n$ is any even number with $n \equiv 0\pmod 8$.

\begin{thm}\label{LastCase}
Let $n$ be an even integer with $n \equiv 0\pmod 8$. Let $x \in \mathsf{S}_{2^k}$ be a picky $2$-element of Type II and let $P_{n}\in \Syl_2(\mathsf{S}_{n})$ be the unique Sylow $2$-subgroup containing $x$. Then there exists  a map 
$$\Gamma:\Irr^x(\mathsf{S}_{n})\rightarrow \Irr^x(P_n) $$
satisfying the following properties.
\begin{itemize}
\item [(I)] $\Gamma(\chi)(1)_2=\chi(1)_2$ for every $\chi\in \Irr^x(\mathsf{S}_{n})$.

\item[(II)]  $\Gamma(\chi)(x)=\pm \chi(x)$ for every $\chi\in \Irr^x(\mathsf{S}_{n})$.
\end{itemize}
Moreover, for $H\in \{\mathsf{S}_{n},P_{n}\}$ we have 
\begin{itemize}

\item[a)]If  $\chi \in \Irr_{2'}(H)$, then $\chi(x)=\pm1$.

\item[b)] If  $\chi \in \Irr^x(H) \setminus\Irr_{2'}(H)$, then $\chi(x)=\pm2$.
\end{itemize}
\begin{proof}
 Let $n=\sum_{i}^f 2^{n_i}$ be the $2$-adic expansion of $n$. Since $\Sub_{\mathsf{S}_{n}}(x)=\mathbf{N}_{\mathsf{S}_{n}}(P_{n})=P_{n}$, the existence of the bijection follows by Theorem \ref{thmB}. Now, we prove that statements a) and b) hold by induction on $f$.

If $f=1$, then $n=2^k$ for $k \geq 3$ and hence the result follows by Corollary \ref{CasePower2}.

Let us assume that the result holds for $f-1$. Let us write $m=n-2^{n_f}$, $k=n_f$ and $x=zy$ where $z\in \mathsf{S}_{2^k}$ is a $2^{k}$-cycle and $y \in \mathsf{S}_{m}$ is a picky $2$-element of Type II.

Let $\lambda$ be a partition of $n$. If $\lambda$ does not possess any $2^{k}$-hook, then $\chi^{\lambda}(x)=0$ by the Murnaghan--Nakayama rule. It follows that $\chi^{\lambda}\in \Irr^x(\mathsf{S}_n)$ if and only if $\lambda$ possesses a (unique) $2^{k}$-hook, say $\tau$, such that $\chi^{\lambda\setminus \tau}\in \Irr^y(\mathsf{S}_{m})$.

Let $\mu$ be a partition of $m$ such that $\chi^{\mu}\in \Irr^y(\mathsf{S}_{m})$. Applying Lemma \ref{ext2}, we have that, for any $2^{k}$-hook, say $\tau$, there exists a unique partition $\lambda(\tau, \mu)$ of $n$ such that $\tau$ is a $2^{k}$-hook of $\lambda(\tau, \mu)$ and $\lambda(\tau, \mu)\setminus \tau=\mu$. By the previous discussion, if $\chi^{\lambda}\in \Irr^x(\mathsf{S}_n)$, then $\lambda=\lambda(\tau,\mu)$ for a $2^k$-hook $\tau$ and $\chi^{\mu}\in \Irr^y(\mathsf{S}_{m})$. By Lemma \ref{ext2}, we also have that 
$$\chi^{\lambda(\tau,\mu)}(1)_2=\chi^{\mu}(1)_2.$$
In addition, applying the Murnaghan--Nakayama rule, we deduce that
$$\chi^{\lambda(\tau,\mu)}(x)=(-1)^{ht(\tau)}\chi^{\mu}(y).$$

If $\chi^{\mu}\in \Irr_{2'}(\mathsf{S}_{m})$, then $\chi^{\mu}(y)=\pm1$ by induction. Thus, $\chi^{\lambda(\tau,\mu)}\in \Irr_{2'}(\mathsf{S}_{m})$ and $\chi^{\lambda(\tau,\mu)}(x)=\pm1$. Analogously, if $\chi^{\mu}\in \Irr^y(\mathsf{S}_{m})\setminus \Irr_{2'}(\mathsf{S}_{m})$, then $\chi^{\lambda(\tau,\mu)}\in \Irr^x(\mathsf{S}_n)\setminus \Irr_{2'}(\mathsf{S}_n)$ and $\chi^{\lambda(\tau,\mu)}(x)=\pm2$. It follows that the assertions a) and b) hold for $\mathsf{S}_n$. 
\end{proof}
\end{thm}

Now, we work towards proving Theorem \ref{thmA} for $p$-adic elements. For this discussion we will assume that $p\neq 2$ (notice that the $2$-adic case follows from Lemma \ref{TypeI}). Assume  first that $n=ap^k$ with $a \leq p-1$.  By Corollary \ref{hooks} we have that  $|\Irr_{p'}(\mathbf{N}_{\mathsf{S}_{p^k}}(P_{p^k}))|=p^k$. Assume that $$\Irr_{p'}(\mathbf{N}_{\mathsf{S}_{p^k}}(P_{p^k}))=\{\chi_0,\chi_1,\ldots, \chi_{p^k-1}\}$$ is  a fixed labelling for the characters in $\Irr_{p'}(\mathbf{N}_{\mathsf{S}_{p^k}}(P_{p^k}))$.

 Let $\lambda$ be a partition of $n=ap^k$ with $\chi^{\lambda}\in \Irr_{p'}(\mathsf{S}_n)$. Then all entries of $T_C(\lambda)$ are $(0)$ except in the $k$-th row. Thus, $\lambda$ is determined by a set of partitions  $\{\lambda_i\mid 0 \leq i \leq p^k-1\}$ satisfying that $\sum_{i=0}^{p^k-1}|\lambda_i|=a$.  Now, we  work to define a character $\Phi^{\lambda}\in  \Irr_{p'}(\mathbf{N}_{\mathsf{S}_n}(P_n))$ from the partitions $\lambda_i$ with $0\leq i\leq p^k-1$.

We know that $\mathbf{N}_{\mathsf{S}_n}(P_n)=\mathbf{N}_{\mathsf{S}_{p^k}}(P_{p^k}) \wr \mathsf{S}_a$. Let $\phi_0,\ldots,\phi_{a} \in \Irr_{p'}(\mathbf{N}_{\mathsf{S}_{p^k}}(P_{p^k}))$, such that $H_i=\{j \in \{1,\ldots a\}\mid  \phi_j=\chi_i \}$ satisfies that $|H_i|=|\lambda_i|$ for each $0\leq i\leq p^k-1$. Let $L$ be the inertia group of $\phi_1\times\ldots\times \phi_a$ in $\mathbf{N}_{\mathsf{S}_n}(P_n)$. By Theorem  \ref{Mattarei}, we have that
$$L\cong \mathbf{N}_{\mathsf{S}_{p^k}}(P_{p^k})^a\rtimes(\Sym(H_0)\times \ldots \times \Sym(H_{p^k-1}))$$
and that $\phi_1 \times\ldots\times\phi_a$ extends to an irreducible  character $\Psi$ of $L$. By Gallagher's Theorem (see Corollary 6.17 of \cite{Isaacscar}), we have that $\Psi(\chi^{\lambda_0}\times \ldots \times \chi^{\lambda_{p^k-1}})\in \Irr(L)$. Now, we define  $\Phi^{\lambda}$ as $(\Psi(\chi^{\lambda_0}\times \ldots \times \chi^{\lambda_{p^k-1}}))^{\mathbf{N}_{\mathsf{S}_n}(P_n)}$. We know that $\Phi^{\lambda}\in \Irr(\mathbf{N}_{\mathsf{S}_n}(P_n))$ by Clifford correspondence (see Theorem 6.11 of \cite{Isaacscar}) and, clearly, $p$ does not divide $\Phi^{\lambda}(1)$. The map taking $\chi^{\lambda}$ to $\Phi^{\lambda}$ is a bijection from $\Irr_{p'}(\mathsf{S}_{ap^k})$ to $\Irr_{p'}(\mathbf{N}_{\mathsf{S}_{ap^k}}(P_{ap^k}))$.

\begin{thm}\label{SmallCase}

Let $p>2$ be a prime, let $0<a<p$ and  $k \geq 1$ be integers and let $x\in \mathsf{S}_{ap^k}$ be a $p$-adic element. Let  $\{\lambda_i\mid0 \leq i \leq p^k-1\}$ be a set of partitions satisfying that $\sum_{i=0}^{p^k-1}|\lambda_i|=a$ and let $\lambda$ be the unique partition of $ap^k$ such that the $k$-th row of $T_C(\lambda)$ is $\{\lambda_i\mid 0 \leq i \leq p^k-1\}$.  Then 
$$\chi^{\lambda}(x)=\pm\Phi^{\lambda}(x)=\pm\frac{a!}{\prod_{i=0}^{p^k-1}|\lambda_i|!}\prod_{i=0}^{p^k-1}\chi^{\lambda_i}(1)$$
As a consequence, the map $\Phi:\Irr_{p'}(\mathsf{S}_{ap^k})\rightarrow \Irr_{p'}(\mathbf{N}_{\mathsf{S}_{ap^k}}(P_{ap^k}))$  defined by $\Phi(\chi^{\lambda})=\Phi^{\lambda}$ is a bijection satisfying $\Phi(\chi)(x)=\pm\chi(x)$.
\begin{proof}
We begin by proving that $\Phi^{\lambda}(x)=\pm\frac{a!}{\prod_{i=0}^{p^k-1}|\lambda_i|!}(\prod_{i=0}^{p^k-1}\chi^{\lambda_i}(1))$. Let $\theta:=\phi_1\times\ldots\times\phi_a\in \Irr_{p'}(\mathbf{N}_{\mathsf{S}_{p^k}}(P_{p^k}))^a$, where the $\phi_i$ are as above. Clearly, $\theta$ is a component of the restriction of $\Phi^{\lambda}$ to $\mathbf{N}_{\mathsf{S}_{p^k}}(P_{p^k})^a$. Let $\theta_1,\ldots,\theta_t$ be the $\mathbf{N}_{\mathsf{S}_n}(P_n)$-conjugates of $\theta$. We know that $x\in \mathbf{N}_{\mathsf{S}_{p^k}}(P_{p^k})^a$ and hence $$\Phi^{\lambda}(x)=(\prod_{i=0}^{p^k-1}\chi^{\lambda_i}(1))(\sum_{i=1}^t\theta_i(x)),$$ by Clifford's Theorem (see Theorem 6.2 of \cite{Isaacscar}).

 For each $0\leq i\leq p^{k}-1$, $\theta=\phi_1\times\ldots\times\phi_a$ contains $|\lambda_i|$ copies of $\chi_i$. Thus, the number of $\mathbf{N}_{\mathsf{S}_n}(P_n)$-conjugates of $\theta$ is $a!/(\prod_{i=0}^{p^k-1}|\lambda_i|!)$, that is $t=a!/(\prod_{i=0}^{p^k-1}|\lambda_i|!)$. Moreover, for any $1\leq j \leq t$, we have that $\theta_j(x)=\theta(x)$ since the components of both $\theta$ and $\theta_j$ are $\phi_1,\ldots,\phi_a$, up to rearrangement. Now,  if we write $x$ as $x_1\cdots x_a$, where each $x_i$ is a $p^k$-cycle, then  $\theta(x)=\prod_{j=1}^a\phi_j(x_j)=\varepsilon \in \{1,-1\}$, by Lemma \ref{Casepk}. Therefore,
$$\sum_{i=1}^t\theta_i(x)=\pm t=\pm\frac{a!}{\prod_{i=0}^{p^k-1}|\lambda_i|!}$$
and the result follows.

The fact that  $\chi^{\lambda}(x)=\pm\frac{a!}{\prod_{i=0}^{p^k-1}|\lambda_i|!}\prod_{i=0}^{p^k-1}\chi^{\lambda_i}(1)$ is 2.7.32 of \cite{JK}.
\end{proof}
\end{thm}

Let $n=\sum_{i=0}^{k} a_ip^i$ and let $\{\lambda_{ij}\mid0\leq i\leq k, 0\leq j \leq p^i-1\}$ be a set of partitions satisfying that $\sum_{j=0}^{p^i-1}|\lambda_{ij}|=a_i$ for every $0\leq i \leq k$. Let $\lambda$ be the unique partition of $n$ such that $T_C(\lambda)=\{\lambda_{ij}\mid0\leq i\leq k, 0\leq j \leq p^i-1\}$.

For each $t \in\{0,\ldots ,k\}$, we define $\mu_{ij}^{t}$ as $\lambda_{ij}$ if $i=t$ and $0$ otherwise. We also define $\mu_t$ as the unique $p'$-partition of $a_tp^t$ such that $T_C(\mu_t)=\{\mu_{ij}^t\mid0\leq i\leq k, 0\leq j \leq p^i-1\}$. In particular,  $\Phi^{\mu_t}\in \Irr_{p'}(\mathbf{N}_{\mathsf{S}_{a_tp^t}}(P_{a_tp^t}))$. Now, we define 
$$\Phi^{\lambda}=\Phi^{\mu_k}\times\ldots\times \Phi^{\mu_0}\in \Irr_{p'}(\mathbf{N}_{\mathsf{S}_n}(P_n)).$$
We know that a similar correspondence can be found in \cite{Fong}.

\begin{thm}\label{OddCase}
Let $p>2$ be a prime, let $n=\sum_{i=0}^{k} a_ip^i$ with $0\leq a_i\leq p-1$, $a_k\not=0$ and let $x=x_1\cdots x_k$, where $x_i$ is the product of $a_i$ cycles of length $p^i$. 
The map $$\Phi: \Irr_{p'}(\mathsf{S}_n)\rightarrow \Irr_{p'}(\mathbf{N}_{\mathsf{S}_n}(P_n))$$ defined by $\Phi(\chi^{\lambda})=\Phi^{\lambda}$ is a bijection satisfying that
$$\chi^{\lambda}(x)=\pm\Phi(\chi^{\lambda})(x)=\pm\frac{\prod_{t}a_t!\prod_{t,j}\chi^{\lambda_{tj}}(1)}{\prod_{t,j}|\lambda_{tj}|!}.$$
In particular,  $\Irr_{p'}(\mathsf{S}_n)=\Irr^x(\mathsf{S}_n)$ and $\Irr_{p'}(\mathbf{N}_{\mathsf{S}_n}(P_n))=\Irr^x(\mathbf{N}_{\mathsf{S}_n}(P_n))$.
\begin{proof}
We proceed by induction on $k$. For $k=0$, the result is trivial. Assume that the result holds for $k-1$. Let $y=x_1\ldots x_{k-1}$ and let $\gamma=C_{p^k}(\lambda)$. Applying Theorem 2.7.27 of \cite{JK}, we deduce that 
$$\chi^{\lambda}(x)=\pm\chi^{\mu_k}(x_k)\chi^{\gamma}(y).$$
Now, we have that $\gamma$ is a $p'$-partition of $n-a_kp^k$ satisfying that $\gamma_{ij}=\lambda_{ij}$ for every $0\leq i\leq k-1$ and every $0 \leq j \leq p^i-1$. By induction, we have that 
$$\chi^{\gamma}(y)=\pm\prod_{t=0}^{k-1}\chi^{\mu_t}(x_t).$$
We deduce that 
$$\chi^{\lambda}(x)=\pm\prod_{t=0}^k\chi^{\mu_t}(x_t).$$
Moreover, by Theorem \ref{SmallCase}, we have that 
$$\chi^{\mu_t}(x_t)=\pm \Phi^{\mu_t}(x_t) = \pm\frac{a_t!}{\prod_{j=0}^{p^t-1}|\lambda_{tj}|!}\prod_{j=0}^{p^t-1}\chi^{\lambda_{tj}}(1)$$
for each $t \in \{0,\ldots, k\}$. Therefore, the first part follows.

Now, we prove the second part. The fact that $\Irr_{p'}(\mathsf{S}_n)=\Irr^x(\mathsf{S}_n)$ is Proposition \ref{GLLV}. It only remains to prove  that $\Irr^x(\mathbf{N}_{\mathsf{S}_n}(P_n)))=\Irr_{p'}(\mathbf{N}_{\mathsf{S}_n}(P_n)))$. Since $\chi^{\lambda}(x)=\pm\Phi(\chi^{\lambda})(x)$ for every $\chi \in \Irr_{p'}(\mathsf{S}_n)$ and $\mathbf{C}_{\mathsf{S}_n}(x)\subseteq \mathbf{N}_{\mathsf{S}_n}(P_n)$,   we have that 
$$|\mathbf{C}_{\mathbf{N}_{\mathsf{S}_n}(P_n)}(x)|= \sum_{\psi\in \Irr^x(\mathbf{N}_{\mathsf{S}_n}(P_n))}|\psi(x)|^2\geq \sum_{\psi\in \Irr_{p'}(\mathbf{N}_{\mathsf{S}_n}(P_n))}|\psi(x)|^2=\sum_{\chi\in \Irr_{p'}(\mathsf{S}_n)}|\chi(x)|^2=$$
$$\sum_{\chi\in \Irr^x(\mathsf{S}_n)}|\chi(x)|^2= |\mathbf{C}_{\mathsf{S}_n}(x)|=|\mathbf{C}_{\mathbf{N}_{\mathsf{S}_n}(P_n)}(x)|,$$
which implies that $\Irr^x(\mathbf{N}_{\mathsf{S}_n}(P_n)))=\Irr_{p'}(\mathbf{N}_{\mathsf{S}_n}(P_n)))$.
\end{proof}
\end{thm}

Now, we study the case of picky elements of Type III.

\begin{thm}\label{TypeIII}
Let $n \equiv \pm3 \pmod{9}$,   let $x,z \in \mathsf{S}_n$ be picky elements of Type I and III, respectively and let $P_n\in \Syl_3(\mathsf{S}_n)$ such that $x,z\in P_n$. Then $\Irr^z(\mathsf{S}_n)=\Irr_{3'}(\mathsf{S}_n)$ and $\Irr^z(\mathbf{N}_{\mathsf{S}_n}(P_n)))=\Irr_{3'}(\mathbf{N}_{\mathsf{S}_n}(P_n)))$ and there exists a bijection
$$\Gamma: \Irr_{3'}(\mathsf{S}_n)\rightarrow \Irr_{3'}(\mathbf{N}_{\mathsf{S}_n}(P_n))$$
such that $\Gamma(\chi)(x)=\pm\chi(x)$ and $\Gamma(\chi)(z)=\pm\chi(z)$ for every $\chi \in \Irr_{3'}(\mathsf{S}_n)$.
\begin{proof}
If $n\in \{3,6\}$, then the result follows by inspection in GAP \cite{gap}.

Let $n =\sum_{i=1}^{f}a_i3^i$ with $a_1>0$ be the $3$-adic expansion of $n$. Let $k=a_13$ and let $m=n-k$. Now, let $P_n=P_{a_k3^k}\times \cdots\times P_{a_13}$, where $P_{a_ip^i}\in \Syl_3(\mathsf{S}_{a_ip^i})$ and let $P_m=P_{a_k3^k}\times \cdots\times P_{a_23^2}$. Given $1 \leq i \leq k$ we choose $x_i\in P_{a_i3^i}$ such that $x_i$ is the product of $a_i$ cycles of length $3^i$ and we choose $z_1\in P_{a_13}$ a picky element of Type III in $\mathsf{S}_{a_13}$.  Then, we may assume that $x=x_1\cdot x_2 \cdots x_f$ and $z=z_1 \cdot x_2 \cdots x_f$. Let us write $y=x_2 \cdots x_f$. We notice that $y$ is a $3$-adic element of $\mathsf{S}_m$

By Theorem \ref{OddCase}, we have that there exists a bijection
$$\Omega: \Irr_{3'}(\mathsf{S}_m)\rightarrow \Irr_{3'}(\mathbf{N}_{\mathsf{S}_m}(P_m))$$
such that $\Omega(\chi)(y)=\pm\chi(y)$ for every $\chi \in \Irr_{3'}(\mathsf{S}_m)$. Moreover, by the case $n\in \{3,6\}$ we know that there exists a bijection
$$\Delta: \Irr_{3'}(\mathsf{S}_k)\rightarrow \Irr_{3'}(\mathbf{N}_{\mathsf{S}_k}(P_k))$$
such that $\Delta(\chi)(x_1)=\pm\chi(x_1)$ and $\Delta(\chi)(z_1)=\pm\chi(z_1)$ for every $\chi \in \Irr_{3'}(\mathsf{S}_k)$.

Let $\lambda$ be a partition of $n$ such that $\chi^{\lambda}(1)_3=1$ and let $T_C(\lambda)=\{\lambda_{ij}\mid 1\leq i \leq f, 0\leq j \leq 3^{i}-1\}$. Let $\gamma_{ij}=\lambda_{ij}$ if $i \geq 2$ and $\gamma_{ij}=0$ if $i=1$ and let $\mu_{ij}=0$ if $i \geq 2$ and $\mu_{ij}=\lambda_{ij}$ if $i=1$. Now, let $\gamma$ be the partition such that $T_C(\gamma)=\{ \gamma_{ij}\mid1\leq i \leq f,  0\leq j \leq 3^{i}-1\}$ and let $\mu$ be the partition such that $T_C(\mu)=\{\mu_{ij}\mid 1\leq i \leq f,  0\leq j \leq 3^{i}-1\}$. We notice that $\gamma$ is a partition of $m$ with $\chi^{\gamma}\in \Irr_{3'}(\mathsf{S}_m)$ and $\mu$ is a partition of $k$ with $\chi^{\mu}\in \Irr_{3'}(\mathsf{S}_k)$. Now, we define
$$\Gamma: \Irr_{3'}(\mathsf{S}_n)\rightarrow \Irr_{3'}(\mathbf{N}_{\mathsf{S}_n}(P_n))$$
by
$$\Gamma(\chi^{\lambda}):=\Omega(\chi^{\gamma})\times\Delta(\chi^{\mu})\in \Irr_{3'}(\mathbf{N}_{\mathsf{S}_n}(P_n)).$$
Now, we have that 
$$\Gamma(\chi^{\lambda})(x)=\Omega(\chi^{\gamma})(y)\Delta(\chi^{\mu})(x_1)=\chi^{\gamma}(y)\chi^{\mu}(x_1)=\chi^{\lambda}(x)$$
and
$$\Gamma(\chi^{\lambda})(x)=\Omega(\chi^{\gamma})(y)\Delta(\chi^{\mu})(z_1)=\chi^{\gamma}(y)\chi^{\mu}(x_1)=\chi^{\lambda}(x).$$
Thus, $\Gamma$ is a bijection satisfying the desired properties.

Now, we prove that $\Irr^z(\mathbf{N}_{\mathsf{S}_n}(P_n))=\Irr_{3'}(\mathbf{N}_{\mathsf{S}_n}(P_n))$. Let $\phi\in \Irr^z(\mathbf{N}_{\mathsf{S}_n}(P_n))$.  Since
$$\mathbf{N}_{\mathsf{S}_n}(P_n)=\mathbf{N}_{\mathsf{S}_m}(P_m)\times \mathbf{N}_{\mathsf{S}_k}(P_k)$$
we have that $\phi=\psi\times \theta$ for $\psi \in \Irr^y(\mathbf{N}_{\mathsf{S}_m}(P_m))$ and $\theta \in \Irr^{z_1}(\mathbf{N}_{\mathsf{S}_k}(P_k))$. From the case $n\in \{3,6\}$ we deduce that $ \Irr^{z_1}(\mathbf{N}_{\mathsf{S}_k}(P_k))= \Irr_{3'}(\mathbf{N}_{\mathsf{S}_k}(P_k))$  and from Theorem \ref{OddCase} we deduce that $\Irr^{y}(\mathbf{N}_{\mathsf{S}_m}(P_m))= \Irr_{3'}(\mathbf{N}_{\mathsf{S}_m}(P_m))$. It follows that $\phi(1)_3=1$. Therefore, $\Irr^z(\mathbf{N}_{\mathsf{S}_n}(P_n))=\Irr_{3'}(\mathbf{N}_{\mathsf{S}_n}(P_n))$. Now, using the bijection $\Gamma$ and reasoning as in Theorem \ref{OddCase}, we deduce that  $\Irr^z(\mathsf{S}_n)=\Irr_{3'}(\mathsf{S}_n)$.
\end{proof}
\end{thm}

Now, we prove Theorem \ref{thmA}.

\begin{proof}[Proof of Theorem \ref{thmA}]
Let $n\geq 1$, let $p$ be a prime and let $P_n\in \Syl_p(\mathsf{S}_n)$. Let $\mathcal{P}\subseteq P_n$ be the set of picky elements contained in $P_n$.

First, let us assume that $p=2$.  In particular,  $\mathbf{N}_{\mathsf{S}_n}(P_n)=P_n$. Assume first that $n$ is odd. In this case,  the unique conjugacy class of picky $2$-elements of $\mathsf{S}_n$ is the conjugacy class of $2$-adic elements. Thus,  $\Irr^{\mathcal{P}}(\mathsf{S}_n)=\Irr^x(\mathsf{S}_n)$ and $\Irr^{\mathcal{P}}(\mathbf{N}_{\mathsf{S}_n}(P_n))=\Irr^x(\mathbf{N}_{\mathsf{S}_n}(P_n))$, where $x$ is a picky $2$-element of Type I. The result follows from Lemma \ref{TypeI}.

Assume now that $n$ is even. Let  $x,y\in P_n$ be picky elements of types I and II, respectively.  By Lemma \ref{TypeI}, we have that $\Irr^x(\mathsf{S}_n)=\Irr_{2'}(\mathsf{S}_n)$ and $\Irr^x(P_n)=\Irr_{2'}(P_n)$. Thus,  $\Irr^{\mathcal{P}}(\mathsf{S}_n)=\Irr^y(\mathsf{S}_n)$ and $\Irr^{\mathcal{P}}(P_n)=\Irr^y(P_n)$. If $n \not \equiv 0 \pmod 8$, then the result follows by combining Lemmas \ref{TypeI} and \ref{8notdivides}. Thus, we may assume that $n \equiv 0 \pmod 8$. Let 
$$\Gamma:\Irr^y(\mathsf{S}_n)\rightarrow \Irr^y(P_n)$$ be a bijection provided by Theorem \ref{LastCase}. Since $\Gamma(\chi)(1)_2=\chi(1)_2$ for every $\chi \in \Irr^y(\mathsf{S}_n)$, we have that $\Gamma(\Irr^x(\mathsf{S}_n))=\Gamma(\Irr_{2'}(\mathsf{S}_n))=\Irr_{2'}(P_n)=\Irr^x(P_n)$. Moreover, if $\chi\in \Irr^y(\mathsf{S}_n)$, then $\Gamma(\chi)(x)=\pm\chi(x)$ by Lemma \ref{TypeI} and $\Gamma(\chi)(y)=\pm\chi(y)$ by Theorem \ref{LastCase}. Thus, the result follows in this case.

Assume now that $p=3$ and $n \equiv \pm 3\pmod{9}$. Let $x,z\in P_n$ be picky elements of types I and III, respectively. By Theorem \ref{TypeIII} and Lemma \ref{TypeI}, we deduce that $\Irr^z(\mathbf{N}_{\mathsf{S}_n}(P_n))=\Irr^x(\mathbf{N}_{\mathsf{S}_n}(P_n))=\Irr_{3'}(\mathbf{N}_{\mathsf{S}_n}(P_n))$ and $\Irr^z(\mathsf{S}_n)=\Irr^x(\mathsf{S}_n)=\Irr_{3'}(\mathsf{S}_n)$. It follows that $\Irr^{\mathcal{P}}(\mathbf{N}_{\mathsf{S}_n}(P_n))=\Irr_{3'}(\mathbf{N}_{\mathsf{S}_n}(P_n))$ and $\Irr^{\mathcal{P}}(\mathbf{N}_{\mathsf{S}_n}(P_n))=\Irr_{3'}(\mathbf{N}_{\mathsf{S}_n}(P_n))$. Moreover, by Theorem \ref{TypeIII}, there exists a bijection 
$$\Gamma: \Irr_{3'}(\mathsf{S}_n)\rightarrow \Irr_{3'}(\mathbf{N}_{\mathsf{S}_n}(P_n))$$
such that $\Gamma(\chi)(x)=\pm\chi(x)$ and $\Gamma(\chi)(z)=\pm\chi(z)$ for every $\chi \in \Irr_{3'}(\mathsf{S}_n)$.

Finally, let us assume that either $p>3$ or  that $p=3$ and $n \not \equiv \pm3 \pmod{9}$. Then the unique conjugacy class of picky $p$-elements of $\mathsf{S}_n$ is the conjugacy class of $p$-adic elements.  Therefore,  $\Irr^{\mathcal{P}}(\mathsf{S}_n)=\Irr^x(\mathsf{S}_n)$ and $\Irr^{\mathcal{P}}(\mathbf{N}_{\mathsf{S}_n}(P_n))=\Irr^x(\mathbf{N}_{\mathsf{S}_n}(P_n))$, where $x\in P$ is a picky $p$-element of Type I.  By Theorem \ref{OddCase}, we have that $\Irr^x(\mathsf{S}_n)=\Irr_{p'}(\mathsf{S}_n)$ and $\Irr^x(\mathbf{N}_{\mathsf{S}_n}(P_n))=\Irr_{p'}(\mathbf{N}_{\mathsf{S}_n}(P_n))$. Now, we have that the map $\Phi$ in Theorem \ref{OddCase} is a bijection between $\Irr^x(\mathsf{S}_n)$ and $\Irr^x(\mathbf{N}_{\mathsf{S}_n}(P_n))$ satisfying that $\chi^{\lambda}(x)=\pm\Phi(\chi^{\lambda})(x)$ and that $\chi^{\lambda}(1)_p=1=\Phi(\chi^{\lambda})(1)_p$. The result follows in this case.
\end{proof}

\renewcommand{\abstractname}{Acknowledgements}
\begin{abstract}
The author would like to thank Gunter Malle, Attila Maróti and Alexander Moretó for many helpful conversations during the preparation of this work. Parts of this work were done when the author was visiting the Alfréd Rényi Institute of Budapest. He thanks the Institute for its hospitality.
\end{abstract}

\end{document}